\documentclass[12pt]{amsart}
\usepackage[height=21.5cm,left=2.5cm,right=2.5cm]{geometry}
\usepackage{graphicx}
\usepackage{pb-diagram}
\usepackage{colordvi}
\usepackage{graphics,eepic}
\usepackage[all]{xy}
\usepackage{mathrsfs}
\usepackage[centertags]{amsmath}
\usepackage{amsfonts}
\usepackage{amssymb}
\usepackage{amsthm}
\newcommand{\pes}[2]{\langle #1,#2\rangle}

\newtheorem{theorem}{Theorem}
\theoremstyle{plain}

\newtheorem{corollary}[theorem]{Corollary}

\newtheorem{definition}[theorem]{Definition}

\newtheorem{lemma}[theorem]{Lemma}

\newtheorem{problem}[theorem]{Problem}
\newtheorem{proposition}[theorem]{Proposition}
\newtheorem{remark}[theorem]{Remark}

\numberwithin{equation}{section}

\newcommand{\pc}{\mathbb{P}(\C^{n})}

\newcommand{\V}{\mathcal{V}}
\newcommand{\Vlin}{\mathcal{V}^{lin}}
\newcommand{\mnc}{\C^{n\times n}}
\newcommand{\mnn}{\C^{(n-1)\times n}}

\newcommand{\W}{\mathcal{W}}
\newcommand{\Un}{\mathcal{U}_n}

\newcommand{\Pn}{\P(\C^{n})}

\def\C{{\mathbb C}}

\def\R{{\mathbb R}}

\def \P{\mathbb P}

\def\S{\mathbb{S}}
\def\CS{\mathcal{S}}

\def\E{\mathbb{E}}
\def\G{\mathcal{G}}

\def\CU{\mathcal{U}}
\def\CV{\mathcal{V}}

\def\CW{\mathcal{W}}
\def\CM{\mathcal{M}_{n}}

\def\NN{{\mathcal{N}}}

\def\Esp{\mathbb E}

\newcommand{\pil}{\pi^{lin}}
\newcommand{\phil}{\phi^{lin}}

\newcommand{\bi}{\begin{itemize}}
\newcommand{\ei}{\end{itemize}}
\newcommand{\bd}{\begin{description}}
\newcommand{\ed}{\end{description}}
\newcommand{\beq}{\begin{equation}}
\newcommand{\eeq}{\end{equation}}
\newcommand{\beqn}{\begin{eqnarray}}
\newcommand{\eeqn}{\end{eqnarray}}
\newcommand{\beqna}{\begin{eqnarray*}}
\newcommand{\eeqna}{\end{eqnarray*}}

\def\XXint#1#2#3{{\setbox0=\hbox{$#1{#2#3}{\int}$}
\vcenter{\hbox{$#2#3$}}\kern-.5\wd0}}

\begin{document}
\title{{\large\texttt{(Rough Draft)}} Average polynomial time for eigenvector computations}

\author[Armentano]{Diego Armentano}
\address{Centro de Matem\'atica, Universidad de la Rep\'ublica. Montevideo, Uruguay}
\email{diego@cmat.edu.uy}
\thanks{First author was partially supported by CSIC, Uruguay}

\author[Beltr\'an]{Carlos Beltr\'an}
\address{Departmento de Matem\'aticas, Universidad de Cantabria. Santander, Spain}
\email{carlos.beltran@unican.es}
\thanks{Second author was partially suported by 
the research project MTM2010-16051 from Spanish Ministry of Science MICINN}

\author[Shub]{Michael Shub}
\address{City University of New York}
\email{shub.michael@gmail.com}
\begin{abstract}
We describe two algorithms for the eigenvalue, eigenvector problem which, on input a Gaussian matrix with complex entries, finish with probability $1$ and in average polynomial time.
\end{abstract}
\maketitle
\section{Introduction}
\footnote{While this manuscript was in the final step of its preparation we learnt of a similar work which has been carried out by Peter Burgisser and Felipe Cucker}
In this paper we exhibit two algorithms $a)$ and $b)$ which take as input a general $n\times n$ complex matrix $A$. Algorithm $a)$ is deterministic while $b)$ is randomized. Here are the most important properties of these algorithms:
\begin{enumerate}
\item Both algorithms $a)$ and $b)$ terminate with probability $1$.
\item Algorithm $a)$ is deterministic and outputs an approximation to the $n$ distinct (eigenvalue,eigenvector) pairs of $A$ while algorithm $b)$ is randomized and outputs an approximation to one (eigenvalue,eigenvector) pair of $A$.
\item The average running time is polynomial: $O(n^9)$ for algorithm $a)$ and $O(n^9)$ for algorithm $b)$, but we pose a question that, if solved, would improve the bound of algorithm $b)$ to $O(n^6)$.
\end{enumerate}
Moreover both algorithms, which employ a variant of Newton's method, are robust with respect to input and round--off error. With a little extra work we expect that a precision sufficient to control the errors may be calculated at each step of the algorithms as is done in polynomial system root finding (see \cite{BeltranLeykin2011,DedieuMalajovichShub}) and as in the polynomial system case the extra precision required at a step of the algorithm comes close to degeneracy does not change the average polynomial nature of the cost.

The algorithms we analyze are not among the usual algorithms used in numerical linear algebra. They are homotopy or continuation methods. These methods have been considered by \cite{Chu1984,LiSauer1987,LiZengCong1992,LiZeng1992} but we do not know of any serious attempt to implement them and so we do not know how they would perform in practice.

There is a precedent \cite{ArmentanoCucker2014} which describes a randomized algorithm with properties similar to our algorithm $b)$ in the case of Hermitian matrices. On the other hand we do not know if any of the standard numerical linear algebra algorithms satisfy $1),2),3)$ above. For example:
\begin{itemize}
\item The unshifted QR algorithm satisfies $1)$ but is probably infinite average cost if approximates to the eigenvectors are to be output (see \cite{Kostlan1988}).
\item The QR algorithm with Rayleigh Quotient shift fails for open sets of real input matrices (see \cite{BattersonSmillie1989,BattersonSmillie1990}), so at least in the real context it does not satisfy $1)$.
\item We do not know if the Francis (double) shift algorithm converges generally on real or complex matrices, nor an exact estimate of its average cost.

Algorithms which output approximate eigenvalues without accompanying approximate eigenvectors might be easier to analyze. The experimental evidence of \cite{PfrangDeiftMenon} for symmetric matrices suggests many of the algorithms in use are of average finite cost and even that there is some universality. But eigenvectors are another matter. When the matrices are close to having multiple eigenvalues the condition of the eigenvector tends to infinity. For example, even for $2\times 2$ symmetric matrices, any pair of orthogonal vectors $(a,b)$ and $(-b,a)$ are the eigenvectors of a matrix
\[
\begin{pmatrix}
1+\varepsilon_1 &\varepsilon_3\\\varepsilon_3&1+\varepsilon_2
\end{pmatrix}
\]
for $|\varepsilon_i|$, $1=1,2,3$, arbitrarily small.

The new points in our analysis are Theorem \ref{thm:mu2average} on the expected value of the condition number squared for the eigenvalue, eigenvector problem and Theorem \ref{th:mainrandom} on how to choose an $n\times n$ matrix $A$, and eigenvalue pair $(\lambda,v)$ of $A$ at random. From these two results our paper follows the methodology laid down in \cite{BeltranPardoFLH} and \cite{BurgisserCucker2009} as adapted to the eigenvalue problem using \cite{Armentano2014}.
\end{itemize}

\section{Description of the main results}
\subsection{Complex random variables preliminaries}\label{sec:rv}

We say that the complex random variable $\eta$ is standard Gaussian, and we write $\eta\sim\NN_\C(0,1)$, if its real and imaginary parts are independent centered Gaussian real random variables with variance $1/2$. That is, its density with respect to the Lebesgue measure on the complex plane is
\begin{equation}
p(z)=\frac{1}{\pi}e^{-|z|^2}. 
\end{equation}
More generally, the normal distribution $\NN_{\C}(\widehat{\eta},\sigma^2)$, with mean $\widehat{\eta}$ and variance $\sigma^2$, has the density
\begin{equation}
p(z)=\frac{1}{\pi\sigma^2}e^{-\frac{|z-\widehat{\eta}|^2}{\sigma^2}}. 
\end{equation}
Given a real valued function $\varphi$, we denote by $\E_{\xi\sim \mathcal{D}}(\varphi(\xi))$ the expected value of the random variable $\varphi(\xi)$, where $\xi$ is $\mathcal{D}$-distributed. Similarly, we use the letter $\P$ for probability (do not confuse with $\pc$ for projective space!). We will skip the sub-index $\xi\sim \mathcal{D}$ when the distribution is understood.

It is easy to check that if $\eta\sim  \NN_{\C}(\widehat{\eta},\sigma^2)$ and $\xi\sim \NN_{\C}(\widehat{\xi},\sigma^2)$ are independent, then
$$
\Esp{\eta}=\widehat \eta;\quad
\Esp\left[{|\eta|^2}\right]=\sigma^2+|\widehat \eta|^2;
\quad \Esp(\eta\, \overline\xi)=\widehat{\eta}\, \overline{\widehat{\xi}}.
$$

\bigskip

Let $A=((a_{j,k}))$ be an $n\times n$ Gaussian complex random matrix centered at $\widehat A=((\widehat{a}_{j,k}))$, that is, $a_{j,k}$ are independent $\NN_{\C}(\widehat{a}_{j,k},\sigma^2)$. For short, let us denote $A\sim \mathcal{G}(n,\sigma^2)_{\widehat A}$, (and $\mathcal{G}(n)$ when centered at the zero matrix and $\sigma=1$). We consider a similar notation $\G(m\times n,\sigma^2)_{\hat A}$ for nonsquare matrices. The density function of $\mathcal{G}(n,\sigma^2)_{\widehat A}$ is
\[
p(A)=\frac{1}{(\sigma^2 \pi)^{n^2}}e^{\frac{-\|A-\widehat A\|_F^2}{\sigma^2}}
\]
We will also consider the truncated Gaussian distribution $\G_T(n,\sigma^2)_{\widehat{A}}$ in $\C^{n\times n}$ whose density function is
\[
\frac{\chi_{\|A-\widehat{A}\|\leq T}}{P_{T,n,\sigma} (\sigma^2\pi)^{n^2}}e^{-\frac{\|A-\widehat{A}\|_F^2}{\sigma^2}},\quad P_{T,n,\sigma}=\P_{A\sim\G(n,\sigma^2)_{\widehat{A}}}(\|A-\widehat{A}\|_F\leq T).
\]
A well known fact is that $P_{\sqrt{2}\,n,n,1}\geq1/2$, which readily implies for every measurable function $\phi:\C^{n\times n}\rightarrow[0,\infty)$:
\begin{equation}\label{eq:truncated}
\E_{A\sim \G_{2n^2}(n)}(\phi)\leq 2\E_{A\sim \G(n)}(\phi)
\end{equation}
\subsection{What does it mean to approximately compute an eigenpair}
The basis for our method is Newton iteration and the associated concept of approximate eigenvalue, eigenvector pairs, following \cite[Section 1.3]{Armentano2014}.
\begin{definition}
For fixed $A\in\mnc$ let $F_A(\lambda,v)=(\lambda I_n-A)v$. We define the Newton iteration  $N_A:\C\times\pc\rightarrow\C\times\pc$ associated to $F_A$ as follows.
\[
N_A\binom{\lambda}{v}=\binom{\lambda}{v}-(DF_A(\lambda,v)\mid_{\C\times v^\perp})^{-1}F_A(\lambda,v)=\binom{\lambda-\dot \lambda}{v-\dot v},
\]
where
\[
\dot v=(\pi_{v^\perp}(\lambda I_n-A)\mid_{v^\perp})^{-1}\pi_{v^\perp}(\lambda I_n-A)v,\quad \dot{\lambda}=\frac{\langle(\lambda I_n-A)(v-\dot v),v\rangle}{\|v\|^2}.
\]
\end{definition}
The following definition taken from \cite{Armentano2014} is inspired in Smale's $\alpha$-theory \cite{Sm86} and the projective version of Newton's method \cite{Sh93}:
\begin{definition}\label{def:eigenpair}
Given a matrix $A\in\mnc$ and given $(\lambda_0,v_0)\in\C\times\Pn$ we say that $(\lambda,v)$ is an aproximate eigenpair of $A$ with associate (exact) eigenpair $(\lambda_\infty,v_\infty)\in\C\times\Pn$ if the sequence of Newton iterates $(\lambda_{i+1},v_{i+1})=N_A(\lambda_i,v_i)$, $i\geq0$ converges immediately and quadratically to $(\lambda_\infty,v_\infty)$, that is, if
\[
d_{\P^2}((A,\lambda_{i+1},v_{i+1}),(A,\lambda_\infty,v_\infty))\leq \frac{1}{2^{2^k-1}}d_{\P^2}((A,\lambda_0,v_0),(\lambda_\infty,v_\infty)),
\]
where $d_{\P^2}$ is the induced distance in $\P(\CM\times \C)\times\pc$.
\end{definition}
Note that an approximate eigenvalue, eigenvector pair is an excellent output of an algorithm: it rapidly produces a pair with any desired precision. The choice of the distance function $d_{\P^2}$ instead of any other is suggested in \cite{Armentano2014} because it simplifies some computations and it behaves well under scalar multiplication of $(A,\lambda)$ and $v$.
\subsection{Geometrical framework}\label{sec:geometry}
Our presentation strongly relies on the same ideas as those developed in \cite{ShSm93b} for the case of polynomial system solving, most importantly on the concept of solution variety. More exactly, we consider the set
\[
\V=\V_n=\{(A,\lambda,v)\in\mnc\times\C\times\pc: \,(\lambda I_n-A)v=0\}.
\]
Following \cite{Armentano2014}, the solution variety $\V$ is a $n^2$--dimensional smooth submanifold of $\mnc\times\C\times\pc$ and it inherits the Riemannian structure of the ambient space.

The solution variety is equipped with two projections
\begin{equation}\label{eq:pi}
\begin{matrix}\pi:&\V&\rightarrow&\mnc\\&(A,\lambda,v)&\mapsto&A\end{matrix}\qquad  \begin{matrix}\pi_2:&\V&\rightarrow&\pc\\&(A,\lambda,v)&\mapsto&v\end{matrix}
\end{equation}
(although a more natural notation for the second mapping would be $\pi_3$, we prefer to use $\pi_2$ following \cite{Armentano2014}). Note that for $A\in\mnc$, the set $\pi^{-1}(A)$ is a copy of the set of eigenpairs $(\lambda,v)\in\C\times\Pn$ of $A$. Similarly, given $v\in\Pn$, the set $\pi_2^{-1}(v)$ is a copy of the linear subspace of $\mnc\times\C$ consisting of the pairs $(A,\lambda)$ such that $\lambda$ is an eigenvalue of $A$ with eigenvector $v$.

The \emph{subvariety $\W$ of well-posed triples} is the subset of
triples $(A,\lambda,v)\in\V$ for which $D\pi (A,\lambda,v)$ is an
isomorphism. In particular, when $(A,\lambda,v)\in\W$, the projection $\pi$ has a branch of the inverse (locally defined) taking $A\in\mnc$ to $(A,\lambda,v)\in\V$.

Given $(A,\lambda,v)\in\mnc\times\C\times\pc$, let $A_{\lambda,v}$ be the linear operator
on the Hermitian complement $v^\perp$ of $v$ given by 
$$
  A_{\lambda,v}:=\Pi_{v^\perp}(\lambda I_n-A)|_{v^\perp}
$$
where $\Pi_{v^\perp}:\C^n\to{v^\perp}$ is the orthogonal projection. 
Then, one can prove that the set of well-posed triples is given by
\begin{equation}\label{eq:Alv}
    \W=\{(A,\lambda,v)\in\V:\;A_{\lambda,v}\,\mbox{ is invertible}\}.
\end{equation}
Let $\Sigma':=\V \setminus \W$ be the \emph{ill-posed variety}, and
$\Sigma=\pi(\Sigma')\subset\mnc$ be the \emph{discriminant variety},
i.e., the subset of ill-posed inputs.

\begin{remark}
From (\ref{eq:Alv}) it is clear that the subset $\Sigma'$ is the set of triples $(A,\lambda,v)\in \V$
such that $\lambda$ is an eigenvalue of $A$ of algebraic multiplicity
at least~2, and $\Sigma$ is the set of matrices $A\in\mnc$ with multiple eigenvalues, thus an algebraic variety.
\end{remark}


Let $\Un$ be the group of $n\times n$ unitary matrices.  The group
$\Un$ naturally acts on $\pc$. In addition, $\Un$ acts on
$\mnc$ by conjugation (i.e., $U\cdot A:=UAU^{-1}$). These actions define an action on the product space $\mnc\times\C\times\pc$, namely,
\begin{equation}\label{eq:UactionPP}
    U\cdot(A,\lambda,v):= (UAU^{-1},\lambda,Uv).
\end{equation}

\begin{remark}
The varieties $\V$, $\W$, $\Sigma\rq{}$, and $\Sigma$,
are invariant under the action of $\Un$.
\end{remark}

\subsection{Average condition number}
Following \cite{Armentano2014} we define the \emph{condition number} and the \emph{Frobenius condition number} of 
$(A,\lambda,v)\in\CW$ as
\begin{equation*}
    \mu(A,\lambda,v)=\max(1,\|A\|_F 
    \left\|{A_{\lambda,v}}^{-1}\right\|),\quad    \mu_F(A,\lambda,v)=\max(1,\|A\|_F 
        \left\|{A_{\lambda,v}}^{-1}\right\|_F),
\end{equation*}
where $\|\cdot\|$ is the operator norm and $\|\cdot\|_F$ is Frobenius norm. Note that in \cite{Armentano2014} the condition number is only defined for $(A,\lambda,v)\in \CW$. We extend the definition to $\CM\times\C\times\pc$ as follows
\[
 \mu(A,\lambda,v)=\max(1,\|A\|_F\|((I-vv^*)(\lambda I-A))^\dagger\|),
\]
where $^\dagger$ means Moore-Penrose pseudoinverse. If $(A,\lambda,v)\in\CW$, this definition is equivalent to the previous one. The value of $\mu$ and $\mu_F$ is set to $\infty$ if the rank of $(I-vv^*)(\lambda I-A)$ is smaller than or equal to $n-2$, equivalently, if $A_{\lambda,v}$ is not invertible for $(A,\lambda,v)\in\CV$. Note that
\begin{equation}\label{eq:muymuF}
\mu(A,\lambda,v)\leq \mu_F(A,\lambda,v)\leq \sqrt{n-1}\, \mu(A,\lambda,v).
\end{equation}
The following result shows that there exist approximate zeros whenever the condition number is finite.
\begin{proposition}\cite[Prop. 4.5]{Armentano2014}\label{prop:Armentano2014}
Let $c_0=0.0739$. Let $A\in\mnc$ and let $(\lambda,v)$ be an eigenvalue, eigenvector pair of $A$ such that $\mu(A,\lambda,v)<\infty$. If $(\lambda_0,v_0)$ satisfies $d_{\P^2}((A,\lambda_0,v_0),(A,\lambda,v))\leq c_0\mu(A,\lambda,v)^{-1}$ then $(\lambda_0,v_0)$ is an approximate eigenvalue, eigenvector pair of $A$ converging to $(\lambda,v)$ under Newton's iteration.
\end{proposition}
\begin{remark}\label{rem:muscaling}
The condition number $\mu$ is invariant under the action of the unitary group $\Un$, i.e.,
$\mu(UAU^{-1},\lambda,Uv)=\mu(A,\lambda,v)$ for all $U\in\Un$.
Moreover, $\mu$ is scale 
invariant on the first two components. That is, 
$\mu(sA,s\lambda,v)=\mu(A,\lambda,v)$ for all nonzero real $s$. These properties also hold replacing $\mu$ by $\mu_F$.
\end{remark}
Another important property of the condition number is its local Lipschitz constant. More precisely, we have (see Prop. \ref{prop:zetamoves} below for a much more precise version):
\begin{lemma}\label{lem:loclip}
There are universal constants $c,C>0$ with the following property. Let $A_0\not\in\Sigma$. Then, for every $A\in\C^{n\times n}$ such that $\|A-A_0\|\leq c\|A_0\|\mu(A_0,\lambda_0,v_0)^{-2}$ we have
\[
\mu(A,\lambda,v)<C\mu(A_0,\lambda_0,v_0),
\]
where $\lambda,v$ is the eigenvalue, eigenvector pair of $A$ obtained by continuation from an eigenvalue, eigenvector pair $(\lambda_0,v_0)$ of $A_0$. For example, one can take $c=1/200$ and $C=3/2$.
\end{lemma}
The condition number plays a significant role in the study of eigenpair computations: on the one hand, from Proposition  \ref{prop:Armentano2014} the convergence of Newton's iteration is granted in a ball of radius which depends on $\mu(A,\lambda,v)^{-1}$. On the other hand, the sensitivity of numerical computations is higher when the condition number grows. Altogether, high values of the condition number impose severe restrictions on the ability to find approximate eigenpairs. Fortunately, if the entries of $A$ are complex Gaussian, we can prove that the condition number of the eigenpairs of $A$ is not too high. Namely, we have the following result (see Section \ref{sec:mu2average} for a proof) which is our first contribution to the problem.

\begin{theorem}\label{thm:mu2average}
For $A\in\mnc$ with Gaussian $N_\C(0,\sigma^2)$ entries centered at $\hat{A}\in\C^{n\times n}$ (i.e. $A\sim\G(n,\sigma^2)_{\hat{A}}$) we have:
\[
 \E_{A\sim\G(n,\sigma^2)_{\hat{A}}}\left(\frac{1}{n}\sum_{\lambda,v:Av=\lambda v}\frac{\mu^2(A,\lambda,v)}{\|A\|_F^2}\right)\leq  \E_{A\sim\G(n,\sigma^2)_{\hat{A}}}\left(\frac{1}{n}\sum_{\lambda,v:Av=\lambda v}\frac{\mu_F^2(A,\lambda,v)}{\|A\|_F^2}\right)\leq \frac{n}{\sigma^2}.
\]
Moreover, for $A$ chosen with the uniform distribution in the unit sphere $\S(\mnc)$ of $\mnc$ (we just write $A\in \S(\mnc)$), we have:
\[
 \E_{A\in\S(\mnc)}\left(\frac{1}{n}\sum_{\lambda,v:Av=\lambda v}\mu^2(A,\lambda,v)\right)\leq \E_{A\in\S(\mnc)}\left(\frac{1}{n}\sum_{\lambda,v:Av=\lambda v}\mu_F^2(A,\lambda,v)\right)\leq n^3.
\]
\end{theorem}
In the case of Gaussian matrices centered at $0$ the bound in the first formula of Theorem \ref{thm:mu2average} can be divided by $2$. However the proof is large and we do not include it here.

\subsection{A homotopy method}
The homotopy continuation method is one of the classical eigenvalue solvers, see for example \cite{LiSauer1987}. Its simple idea is as follows: given a matrix $A$, take another matrix $A_0$ with a known eigenpair (for example, $A_0$ can be a diagonal matrix). Then, consider a path $t\mapsto A_t\in\mnc$, $t\in[0,a]$, with initial point $A_0$ and endpoint $A$. Denoting by $(\lambda_t,v_t)$ the eigenpair of $A_t$ defined by continuation, we have $A_tv_t-\lambda_tv_t=0$. After differentiating this last expression we obtain an Initial Value Problem and its numerical solution describes an approximation of an eigenpair of the endpoint $A$. The main ingredient for the complexity estimate is the number of homotopy steps, i.e., the number of points in the discretization of the interval $[0,a]$ needed to approximate the solution of the IVP.

Formalizing this idea (i.e. proving that the solution exists and is unique, describing a method for solving the IVP with guarantee of convergence, etc.) and proving that the output is an actual approximate eigenpair in the sense of Definition \ref{def:eigenpair} is a nontrivial task, that has been developed rigorously for the first time in \cite{Armentano2014} following the ideas in \cite{Shub2007}. From \cite[Th. 3 and Prop. 3.14]{Armentano2014} we have:
\begin{proposition}\label{prop:stepsold}
If $B_t$ is given by the great circle in the sphere joining $A_0/\|A_0\|_F$ and $A/\|A\|_F$, i.e.,
\begin{equation}\label{eq:At}
B_t=\frac{A_0}{\|A_0\|_F}\cos t+\frac{A-\R e\langle A_0,A\rangle A_0/\|A_0\|_F^2}{\sqrt{\|A\|_F^2-\R e\langle A_0,A\rangle^2/\|A_0\|_F^2}}\sin t,\quad t\in\left[0,a\right],
\end{equation}
where $a=\arccos\R e\langle \frac{A_0}{\|A_0\|_F},\frac{A}{\|A\|_F}\rangle$ is the distance in the unit sphere $\S(\mnc)$ from $B_0=A_0/\|A_0\|_F$ to $B_a=A/\|A\|_F$, we can bound the number of steps needed by the continuation method to compute an approximate eigenvalue, eigenvector pair of $A$ by
\[
\mathcal{C}(A,A_0,\lambda_0,v_0)=c\int_0^a\mu(B_t,\lambda_t,v_t)^2\,dt,
\]
$c$ a universal small constant and $\lambda_t,v_t$ defined by continuation. The number of arithmetic operations is at most $O(n^3)$ times that quantity.
\end{proposition}
The result in \cite{Armentano2014} is not constructive because some constants remain uncomputed. An effective, algorithmic (yet, possibly not optimal), version of this result is included in Appendix \ref{appendix2}. It is clear in the proof that the upper bound on the number of steps is within a constant factor of the actual value. Long step methods (i. e. methods based on a standard use of routines for IVP approximate solving) should be more efficient than our algorithm, but  we haven't proven theorems about them. An estimate on the average curvature of the solution curves (in the spirit of \cite{DedieuMalajovichShub2005} for the linear programming setting) would be interesting. Small average curvature would indicate that long step methods should be successful.
\begin{lemma}\label{lem:daigual}
For fixed $A_0,A$ consider the path $A_t=(1-t)A_0+tA$ which satisfies $A_1=A$. Then,
\[
\mathcal{C}(A,A_0,\lambda_0,v_0)\leq c\|A_0\|_F\,\|A_1\|_F\int_0^1\frac{\mu(A_t,\lambda_t,v_t)^2}{\|A_t\|_F^2}\,dt,
\]
where $\lambda_t$, $v_t$ are the eigenvalue and eigenvector of $A_t$ defined by continuation.
\end{lemma}
\begin{proof}
Let $B_s$, $s\in[0,a]$, be the arc--lenght parametrized spherical segment from $A_0/\|A_0\|_F$ to $A/\|A\|_F$, where $a=\arccos\R e\langle \frac{A_0}{\|A_0\|_F},\frac{A}{\|A\|_F}\rangle$. Then,
\[
\mathcal{C}(A,A_0,\lambda_0,v_0)=c\int_0^a\mu(B_s,\lambda_{B_s},v_{B_s})^2\,ds=c\int_0^a\mu(B_s,\lambda_{B_s},v_{B_s})^2\|\dot B_s\|_F\,ds,
\]
where $\lambda_{B_s}$ and $v_{B_s}$ are the eigenvalue and eigenvector of $B_s$ defined by continuation. Now, reparametrizing that spherical segment by $C_t=A_t/\|A_t\|$, $t\in[0,1]$, the integral does not change. We thus have
\[
\mathcal{C}(A,A_0,\lambda_0,v_0)=c\int_0^1\mu(C_t,\lambda_{C_t},v_{C_t})^2\|\dot C_t\|\,dt\underset{\text{Rmk. \ref{rem:muscaling}}}{=}c\int_0^1\mu(A_t,\lambda_t,v_t)^2\left\|\frac{d}{dt}\left(\frac{A_t}{\|A_t\|}\right)\right\|\,dt.
\]
Substituting $\dot A_t=A_1-A_0$ and $A_t=(1-t)A_0+tA_1$ in this last formula and with some elementary computations we conclude that
\[
\left\|\frac{d}{dt}\left(\frac{A_t}{\|A_t\|}\right)\right\|\leq\frac{\|A_0\|_F\,\|A_1\|_F}{\|A_t\|_F^2}.
\]
 We have thus proved:
\[
\mathcal{C}(A,A_0,\lambda_0,v_0)\leq c\|A_0\|_F\,\|A_1\|_F\int_0^1\frac{\mu(A_t,\lambda_t,v_t)^2}{\|A_t\|_F^2}\,dt,
\]
as wanted.
\end{proof}
\subsection{Algorithm a)}
In this section we follow the ideas in \cite{BurgisserCucker2009} adapting them to the case of eigenvalue, eigenvector computations. Consider the following algorithm on input $A\in\mnc$:
\begin{enumerate}
\item Choose any matrix $A_0\not\in\Sigma$ such that we know in advance all its eigenvalue, eigenvector pairs $(\lambda_0^{(i)},v_0^{(i)})$, $1\leq i\leq n$.
\item Follow the $n$ homotopy paths given by \eqref{eq:At} using the algorithm of Proposition \ref{prop:stepsold} (see Appendix \ref{appendix2}) starting at each $(A_0,\lambda_0^{(i)},v_0^{(i)})$ and output $n$ approximate eigenvalue, eigenvector pairs $(\lambda^{(i)},v^{(i)})$ of $A$.
\end{enumerate}
Note that the algorithm fails if $A\in\Sigma$ or if some $A_t$ in the homotopy path lies in $\Sigma$. This however happens with probability $0$ on $A$.
\begin{theorem}\label{th:deterministic1}
Algorithm $a)$ performs on the average $O(n^4\mu(A_0)^2)$ homotopy steps when $A$ is drawn from $\G(n)$. Its total average complexity is thus $O(n^7\mu(A_0)^2)$. Moreover, choosing
\[
A_0=Diag(\lambda_1,\ldots,\lambda_n),
\]
where the $\lambda_i$ are the successive points in the hexagonal lattice chosen such that $0=|\lambda_1|\leq\cdots\leq|\lambda_n|$, and $v_0^{(i)}$ the $i$--th vector of the canonical basis, $1\leq i \leq n$, we have $\mu(A_0)^2\leq \sqrt{3}\pi n^2/27+o(n^2)$ and the total average complexity for producing an approximate eigenvalue, eigenvector pair for all eigenvalue, eigenvector pairs of $A\sim\G(n)$ is $O(n^9)$.
\end{theorem}
\begin{remark}
Theorem  \ref{th:deterministic1} suggests the following question which is certainly interesting by itself: which diagonal $A$ matrix with eigenvalues $\lambda_1,\ldots,\lambda_n$ has an optimal (smallest possible) condition number? That is, we search for $\lambda_1,\ldots,\lambda_n\in\C$ minimizing the quantity
\[
\mu(A)^2=\|A\|_F^2\max_{i\neq j}|\lambda_i-\lambda_j|^{-2}=(|\lambda_1|^2+\cdots+|\lambda_n|^2)\max_{i\neq j}|\lambda_i-\lambda_j|^{-2}.
\]
The proposed choice based on a hexagonal lattice gives quite a reasonable candidate for this problem, but it can probably be done  better for selected values of $n$.
\end{remark}
\subsection{Random homotopy}
In this section we follow the ideas in \cite{BeltranPardoFLH} adapting them to the case of eigenvalue, eigenvector computations. Proposition \ref{prop:stepsold} claims that the complexity of path--following methods for the eigenvalue problem essentially depends on the squared condition number along the path. This quantity is usually unknown a priori, but from Theorem \ref{thm:mu2average} we know that, on the average, for matrices with Gaussian entries, it has polynomial value. This suggests that choosing at random initial points is a reasonable strategy; the following result proves this fact.
\begin{theorem}\label{th:randomhomotopy}
Consider the following algorithmic scheme: on input $A\in\mnc$,
\begin{itemize}
\item Choose at random with the uniform distribution a matrix $A_0\sim\G(n)$. Choose at random (with the uniform discrete distribution) one of the eigenpairs $(\lambda_0,v_0)$ of $A_0$.
\item Define the path $B_t$ as in \eqref{eq:At}.
\item Use the algorithm mentioned in Proposition \ref{prop:stepsold} to generate an approximate eigenpair $(\lambda,v)$ of $A$.
\end{itemize}
Then, the expected average number of homotopy steps of this algorithm is at most
\[
\mathbb{E}_{A\sim\G(n),A_0\sim\G(n)}\left(\mathbb{E}_{\lambda_0,v_0:A_0v_0=\lambda_0v_0}\mathcal{C}(A,A_0,\lambda_0,v_0)\right)\leq O(n^3)
\]
and its expected average number of arithmetic operations is at most $O(n^6)$.
\end{theorem}
The idea behind the proof of Theorem \ref{th:randomhomotopy} (see Section \ref{sec:proofofrandomisrandom}) is very simple: from Proposition \ref{prop:stepsold}, for a randomly chosen input $A$ and a random choice $(A_0,\lambda_0,v_0)$ the number of homotopy steps is given by the integral along the great circle joining $A_0/\|A_0\|_F$ and $A/\|A\|_F$ of the square of the condition number. Now, there is no element of $\S(\mnc)$ that is more weighted than any other one in this setting, so the average must be given by a constant times the average of the squared condition number in $\S(\mnc)$ which was computed in Theorem \ref{thm:mu2average}.

It must be noted by the reader that Theorem \ref{th:randomhomotopy} does {\em not} immediately produce an average polynomial time algorithm for computing eigenpairs, since the first step (choosing $A_0$ and then $(\lambda_0,v_0)$ at random) already requires solving an EVP problem, a task that has not been proved to be doable in polynomial time yet! This is a similar situation to that solved in \cite{BePa05e,BeltranPardoFLH}, where a random polynomial system and one of its zeros at random had to be chosen. Following the ideas in those papers, we note that Theorem \ref{th:randomhomotopy} would yield an actual average polynomial time algorithm if we could find some collection of probability spaces $\Omega_n$ and functions $\varphi_n:\Omega_n\rightarrow\CV_n$, $n\geq2$, such that:
\begin{enumerate}
\item Choosing $w\in\Omega_n$ can be done starting on a number of random choices of numbers with the $N_\C(0,1)$ distribution, and performing some arithmetic operations on the results, the total expected running time being bounded by a polynomial in $n$.
\item Given $w\in\Omega_n$, $\varphi_n(w)$ is computable in average polynomial time, that is the expected number of arithmetic operations for computing $\varphi_n(w)$ must be bounded above by a polynomial in $n$.
\item Choosing $w$ at random in $\Omega_n$ and computing $({A}_0,{\lambda}_0,v_0)=\varphi_n(w)$ is equivalent to choosing $A_0\sim\G(n)$ at random and choosing at random $(\lambda_0,v_0)$ such that $Av_0=\lambda_0v_0$. That is, for any measurable mapping $\phi:\CV_n\rightarrow[0,\infty]$ we must have
\begin{equation}\label{eq:wish}
\mathbb{E}_{w\in\Omega_n}\left(\phi(\varphi_n(w))\right)=\mathbb{E}_{A_0\sim\G(n)}\left(E_{\lambda_0,v_0:A_0v_0=\lambda_0v_0}\phi(A_0,\lambda_0,v_0)\right),
\end{equation}
so that we can apply this equality to $\phi(A_0,\lambda_0,v_0)=\mathbb{E}_{A\sim\G(n)}\left(\mathcal{C}(A,A_0,\lambda_0,v_0)\right)$ and apply Theorem \ref{th:randomhomotopy}.
\end{enumerate}
Unfortunately, we are not able to produce a collection of probability spaces $\Omega_n$ and functions $\varphi_n$ as described above. However, we will prove that relaxing \eqref{eq:wish} to the following less restrictive situation is actually possible: instead of demanding the equality in \eqref{eq:wish} we can just demand an inequality where the right--hand term is multiplied by some polynomial in $n$. Moreover, we do not need \eqref{eq:wish} to hold for every measurable function $\phi$ since all the interesting functions for the EVP problem are projective functions, invariant under the action of the unitary group. We can thus relax \eqref{eq:wish} to hold only with a polynomially bounded inequality, and for unitary invariant projective functions. Proving that this can actually be done is or goal now. We start by defining $\Omega_n$ and $\varphi_n$:
\begin{definition}\label{def:OmegayVarphi}
For every $n\geq2$, let
\[
\Omega_n=\C\times \mnn\times \CU_{n-1}\times\C^{n-1}
\]
where $\CU_{n-1}$ is the set of $(n-1)\times(n-1)$ unitary matrices, be endowed with the density
\[
\theta_n=C_n\frac{2n^3}{\pi^{n^2}Vol(\CU_{n-1})}e^{-2n^3|z|^2-\|M\|_F^2-\|w\|^2}\;  \chi_{n|z|\,\|M^\dagger\|_F\leq1},\quad (z,M,U,w)\in\Omega_n,
\]
$C_n$ a normalizing constant chosen to make the total volume equal to $1$. Then, let
\[
\varphi_n(z,M,U,w)=
\left(\begin{pmatrix}
z&w^*\\0&MQ_MU
\end{pmatrix},z,e_1
\right),
\]
where $Q_M\in\C^{n\times(n-1)}$ is defined as the $Q$ factor in the reduced $QR$ decomposition of $M^*$ computed by the standard algorithm using Householder reflections.
\end{definition}
Note that the use of Householder algorithm is fixed by the definition. The choice of the method is not important (one can use, for example, Givens rotations, and the following result still holds). What is important is that {\em some} method is chosen that produces the $Q$ factor in an almost everywhere continuous manner, so that $\varphi_n$ is a.e. continuous. Note also that there is a zero measure set of $M$ such that the Householder reflectors in the standard ordering do not produce an answer. Formally speaking, we should discard that set in the definition of $\Omega_n$ but we trust that the slight abuse of notation will not confuse the reader. Our last main result is:
\begin{theorem}\label{th:mainrandom}
Let $\Omega_n$ and $\varphi_n$ be as in Definition \ref{def:OmegayVarphi} for all $n\geq2$. Then:
\begin{enumerate}
\item Choosing $w\in(\Omega_n,\theta_n)$ can be done by choosing $2n^2-2n+1$ numbers with the $N_\C(0,1)$ distribution, cheking a test which involves the computation of a Moore--Penrose inverse and computing a $QR$ decomposition. The total expected running time is $O(n^3)$.
\item Given $w\in\Omega_n$, $\varphi_n(w)$ can be computed by computing a $QR$ decomposition and performing two matrix multiplications. Its total running time is thus $O(n^3)$.
\item For any unitarily invariant, projective, a.e. continuous mapping $\phi:\V_n\to[0,\infty]$,
\begin{equation}\label{eq:wish2}
\mathbb{E}_{w\in\Omega_n}\left(\phi(\varphi_n(w))\right)\leq 22\,n^3\mathbb{E}_{A_0\sim\G(n)}\left(\Sigma_{\lambda_0,v_0:A_0v_0=\lambda_0v_0}\phi(A_0,\lambda_0,v_0)\right).
\end{equation}
\end{enumerate}
\end{theorem}
Note that \eqref{eq:wish2} can be understood as follows: let $m_1$ be the push--forward measure of $\varphi_n$ in $\V$ and let $m_2$ be the measure in $\V$ given by
\[
m_2(X)=\E_{A_0\sim\G(n)}\left(\sharp\{(\lambda_0,v_0):A_0v_0=\lambda_0v_0,(A_0,\lambda_0,v_0)\in X\}\right),\quad X\in\V,X {\text{ measuable}}.
\]
Then,
\[
m_1\leq 22\,n^3m_2.
\]
\begin{problem}
Describe an alternative collection $(\Omega_n,\varphi_n)$ which satisfies a sharper version of \eqref{eq:wish2}, with a constant in the place of $22n^3$. This would improve the running time of Algorithm $b)$ below by a factor of $O(n^3)$.
\end{problem}
We are now prepared to describe our algorithm $b)$:
\subsection{Algorithm $b)$}
Consider the following algorithm on input $A\in\mnc$:
\begin{enumerate}
\item Choose $w\in(\Omega_n,\theta_n)$, compute $\varphi_n(w)=(A_0,\lambda_0,v_0)$ (note that $v_0=e_1$).
\item Follow the homotopy path given by \eqref{eq:At} using the algorithm of Proposition \ref{prop:stepsold} (see Appendix \ref{appendix2}) and output $(\lambda,v)$ an approximate eigenpair of $A$.
\end{enumerate}
From Theorem \ref{th:mainrandom}, the expected running time of the first part of the algorithm is $O(n^3)$. Moreover, the expected number of homotopy steps in the second part is
\[
S=\mathbb{E}_{A\sim\G(n),w\in\Omega_n}\left(\mathcal{C}(A,\varphi_n(w))\right)=\mathbb{E}_{w\in\Omega_n}\left(\mathbb{E}_{A\sim\G(n)}\left(\mathcal{C}(A,\varphi_n(w))\right)\right),
\]
where $\mathcal{C}(A,\varphi_n(w))$ is as in Proposition \ref{prop:stepsold}. From \eqref{eq:wish2}, we have
\[
S\leq 22\,n^3\mathbb{E}_{A_0\sim\G(n)}\left(\Sigma_{\lambda_0,v_0:A_0v_0=\lambda_0v_0}\mathbb{E}_{A\sim\G(n)}\left(\mathcal{C}(A,A_0,\lambda_0,v_0)\right)\right)\underset{\text{Th. \ref{th:randomhomotopy}}}{\leq} O(n^6).
\]
The expected running time of the second part of the algorithm is equal to the expected number of homotopy steps, multiplied by the number of arithmetic operations needed to perform one homotopy step, that is $O(n^3)$, so the total average running time of the algorithm is bounded by $O(n^9)$, a polynomial in $n$ as claimed. Note that this is just an upper bound: the practical performance of the algorithm might be much better.

\section{Proof of $(1)$ and $(2)$ in Theorem \ref{th:mainrandom}}
Note that $(2)$ is trivial. We thus prove $(1)$. The procedure we suggest to chose $w\in\Omega_n$ at random is the following (note that each step requires $O(n^3)$ arithmetic operations or random choices):
\begin{enumerate}
\item Choose $y\sim \G(1)$ and let $z=y/((2n^3)^{1/2})$ (this produces $z\sim\G(1,(2n^3)^{-1})$). Choose $M\sim\G((n-1)\times n)$. If $n|z|\,\|M^\dagger\|_F>1$ then discard $z,M$ and repeat this step.
\item Choose $w\in\G(1\times(n-1))$.
\item Choose $B\in\G(n-1)$ and let $U$ be the $Q$ factor in the $QR$--decomposition of $B$, then multiply $Q$ by the diagonal matrix with entries $r_{ii}/|r_{ii}|$ where the $r_{ii}$ are the diagonal elements of the $R$ factor. This produces a unitary matrix $U$ uniformly distributed in $\CU_{n-1}$, see for example \cite{Mezzadri2007}.
\end{enumerate}
The only subtle point is that step $(1)$ might have to be repeated an arbitrary number of times. We now prove that this is not a troubling problem. Let us denote by $P(X)$ the probability that a certain event $X$ happens. The expected number of times that step (1) will be repeated is $\sum_{k=1}^\infty k P(\text{step $k$ is reached})$. Thus, the expected number of times that we will have to choose $z\sim\G(1,(2n^3)^{-1})$ and $M\sim\G((n-1) \times n)$ is given by
\[
\sum_{k=1}^\infty k P\left(z,M:\|M^\dagger\|_F\geq \frac{1}{n|z|}\right)^{k-1}=
\frac{1}{\left(1-P\left(z,M:\|M^\dagger\|_F\geq \frac{1}{n|z|}\right)\right)^2}.
\]
The expected value of $\|M^\dagger\|_F$ can be computed exactly:  $\E_{M\sim\G((n-1)\times n)}(\|M^\dagger\|_F^2)=n-1$ (see Corollary \ref{cor:exactpinvM}). From Markov's inequality we then have for any $T>0$
\begin{equation}\label{eq:probmalo}
P(\|M^\dagger\|_F\geq T)=P(\|M^\dagger\|_F^2\geq T^2)\leq \frac{n-1}{T^2}\leq \frac{n}{T^2}.
\end{equation}
Hence,
\begin{equation}\label{eq:prob}
P\left(z,M:\|M^\dagger\|\geq \frac{1}{nt}\right)=\frac{2n^3}{\pi}\int_{z\in\C}P\left(M:\|M^\dagger\|_F\geq \frac{1}{n|z|}\right)e^{-2n^3|z|^2}\,dz\leq 
\end{equation}
\[
\frac{2n^3}{\pi}\int_{z\in\C}n^3|z|^2e^{-2n^3|z|^2}\,dz=\frac{1}{2}.
\]
That is, the expected number of times that $(1)$ will be executed is at most
\[
\frac{1}{\left(1-\frac{1}{2}\right)^2}=4.
\]
\section{Integration in the solution variety}\label{sec:interationV}

Given $\varphi:\V\to[0,\infty]$, we define $\hat\varphi:(\mnc\setminus\Sigma)\to\R$ by
$$
\widehat\varphi(A)=\sum_{\lambda,v:Av=\lambda v}\varphi(A,\lambda,v).
$$
The smooth coarea formula (a modern classical formula due to Federer \cite{Fe69}, see \cite[p. 245]{BlCuShSm98} for a smooth version) can be used to deintegrate functions defined on $\CV$ using the projections \eqref{eq:pi}. More exactly, for every measurable nonnegative function defined on $\CV$ we have:
\begin{equation}\label{eq:formulafordeterministic}
\int_{A\in\C^{n\times n}}\widehat{\phi}(A)\,dA=\int_{v\in\pc}\int_{(A,\lambda):Av=\lambda v}\phi(A,\lambda,v)\frac{NJ(\pi)(A,\lambda,v)}{NJ(\pi_2)(A,\lambda,v)}\,d(A,\lambda)\,dv,
\end{equation}
where $NJ$ means Normal Jacobian (i.e. the determinant of the Jacobian restricted to the orthogonal of its kernel). If additionally the function $\phi$ is unitarily invariant, then the inner integral in the right--hand side does not depend on $v$, so we can take $v=e_1$ and we have:
\[
\int_{A\in\C^{n\times n}}\widehat{\phi}(A)\,dA=vol(\pc)\int_{(A,\lambda):Ae_1=\lambda e_1}\phi(A,\lambda,e_1)\frac{NJ(\pi)(A,\lambda,e_1)}{NJ(\pi_2)(A,\lambda,e_1)}\,d(A,\lambda).
\]
The set $\{{(A,\lambda):Ae_1=\lambda e_1}\}$ can be parametrized by $(\lambda,B,w)\rightarrow(A,\lambda)$ where $\lambda\in\C,B\in\C^{(n-1)\times(n-1)},w\in\C^{n-1}$ and
\begin{equation}\label{eq:A}
A=\begin{pmatrix}
\lambda&w^*\\0&B
\end{pmatrix}.
\end{equation}
The Jacobian of the parametrization is $2$. We thus have:
\[
\int_{A\in\C^{n\times n}}\widehat{\phi}(A)\,dA=vol(\pc)\int_{\lambda\in\C,B\in\C^{(n-1)\times(n-1)},w\in\C^{n-1}}2\phi(A,\lambda,e_1)\frac{NJ(\pi)(A,\lambda,e_1)}{NJ(\pi_2)(A,\lambda,e_1)}\,d(\lambda,B,w),
\]
where $A$ in the right--hand term is given by \eqref{eq:A}. Finally, the quotient of normal jacobians in the formula above is equal to $|\det(B-\lambda I_{n-1})|^2/2$ (see Appendix \ref{sec:nj}) and we conclude:
\begin{equation}\label{eq:coareap2}
\int_{A\in\C^{n\times n}}\widehat{\phi}(A)\,dA=vol(\pc)\int_{\lambda\in\C,B\in\C^{(n-1)\times(n-1)},w\in\C^{n-1}}\phi(A,\lambda,e_1)|\det(B-\lambda I_{n-1})|^2\,d(\lambda,B,w),
\end{equation}
where $A$ in the right--hand term is given by \eqref{eq:A}. After computing the constants involved, we have proved the following result. 
\begin{proposition}\label{prop:trick}
Let $\phi:\V\to[0,\infty]$ be a unitarily invariant mesurable function. Then,
$$
\E_{A\sim\G(n)}(\widehat\phi(A))=\frac{1}{\Gamma{(n)}}
\E_{\lambda}\E_{w}\E_{B} 
\left(
\phi\left(\begin{pmatrix}
         \lambda & w\\
         0 & B
        \end{pmatrix},\lambda,e_1\right)
\cdot|\det(B -\lambda I_{n-1})|^2
\right)
$$
where $\lambda,w$ and $B$ are independent $\G(1)$, $\G{1\times (n-1) }$, and $\G(n-1)$, respectively.
\end{proposition}

\section{Proof of Theorem \ref{thm:mu2average}}\label{sec:mu2average}
We begin with the following result.
\begin{proposition}\label{prop:smooth} The following inequality holds
$$
\Esp_{A\sim \mathcal{G}(m,\sigma^2)_{\widehat A}} (\|A^{-1}\|_F^2\,|\det(A)|^2))\leq \frac{m}{\sigma^2}\, \Esp_{A\sim \mathcal{G}(m,\sigma^2)_{\widehat A}} (|\det(A)|^2)).
$$
Furthermore, the equality holds if and only if $\widehat{A}=0$. In particular, 
$$
\Esp_{A\in\mathcal{G}(m,\sigma^2)} (\|A^{-1}\|_F^2\,|\det(A)|^2))=m!\,m \sigma^{2m-2}.
$$
\end{proposition}

\begin{proof}
Expanding the determinant of $A$ by the $k$th column we have
 $$
 \det(A)=\sum_{j=1}^m (-1)^{j+k}{a}_{j,k} \det{A}^{j,k},
 $$
where $A^{j,k}$ denotes the matrix that results from the matrix $A$ by removing the $j$th row and $k$th column. 
Hence,
 $$
 |\det(A)|^2=\det(A)\overline{\det(A)}=\sum_{j,j'=1}^m (-1)^{j+j'+2k}{a}_{j,k}\,\overline{{a}_{j',k}}\, \det{A}^{j,k}\,\overline{\det{A}^{j',k}},
 $$
Observe that the random variables ${a}_{j,k}$ and ${a}_{j',k}$ are independent of $\det{A}^{j,k}$ and $\det{A}^{j',k}$. 
Then, 
$$
 \Esp_{A\sim \mathcal{G}(m,\sigma^2)_{\widehat A}}|\det(A)|^2=\sum_{j,j'=1}^m (-1)^{j+j'+2k}\Esp({a}_{j,k}\,\overline{{a}_{j',k}} )\,\Esp(\det{A}^{j,k}\,\overline{\det{A}^{j',k}}),
 $$
Now observe that
\[ \Esp({a}_{j,k}\,\overline{{a}_{j',k}} )= \left\{ \begin{array}{ll}
          \widehat{a}_{j,k}\,\overline{\widehat{a}_{j',k}} & \mbox{if $j\neq j'$};\\
        \sigma^2+|\widehat{a}_{j,k}|^2 & \mbox{otherwise}.\end{array} \right. \] 
Then we conclude,
 \begin{equation}\label{saconstant2}
 \Esp_{A\sim \mathcal{G}(m,\sigma^2)_{\widehat A}}|\det(A)|^2=\Esp|\det([A;k;\widehat{A}_k])|^2+\sigma^2\sum_{j=1}^m \Esp |\det{A}^{j,k}|^2 ,\qquad (k=1,\ldots,m),
 \end{equation}
 where $[A;k;\widehat{A}_k]$ is the matrix formed by replacing the (random) $k$th column of $A$ by the (deterministic) $k$th column of  $\widehat{A}$.
 
 On the other hand, from a direct application of Cramer's Rule and (\ref{saconstant2}), we conclude
 $$
 \sigma^2\Esp_{A\sim \mathcal{G}(m,\sigma^2)_{\widehat A}}\|A^{-1}\|^2_F\,|\det(A)|^2=\sigma^2\sum_{j,k=1}^m \Esp|\det{A^{j,k}}|^2=m\, \Esp|\det(A)|^2- \sum_{k=1}^m\Esp|\det([A;k;\widehat{A}_k])|^2,
 $$
and the fist claim of the proposition follows. Moreover, when $\widehat A=0$, from (\ref{saconstant2}) and the fact that the matrices $A^{j,k}$ are $\mathcal{G}(m-1,\sigma^2)$-distributed, working by induction the equality
$$
\Esp_{A\in\mathcal{G}(m,\sigma^2)}|\det(A)|^2=\sigma^{2m}m!.
$$
holds, and the second claim of the proposition follows.
 \end{proof}

From \eqref{eq:formulafordeterministic} and from Lemma \ref{lem:quotient-NJ}, for $\alpha\in\{0,1\}$ we have
\begin{equation}\label{eq:citame}
\Esp_{A\sim\G(m)_{\widehat{A}}}\left(\frac{\widehat{\mu_F^{2\alpha}}(A)}{\|A\|^{2\alpha}}\right)=
\int_{v\in\mathbb{P}(\C^n)} \frac{e^{-\|\hat{y}_v\|^2}}{\pi^{n-1}}\Esp_\lambda\Esp_{w}\Esp_{B\sim\mathcal{G}(m-1)_{\widehat B}}(\|(B-\lambda I_{n-1})^{-1}\|_F^{2\alpha}\,|\det(B-\lambda I_{n-1})|^2)\,dv,
\end{equation}
where $\hat{y}_v=\pi_{v^\perp}\hat A v$, and $\lambda\in\C$, $w\in\C^{n-1}$, and $B\in\C^{(n-1)\times (n-1)}$ are independent Gaussian random variables centered at 
$$
\widehat{\lambda}:=\frac{\pes{\widehat{A}v}{v}}{\|v\|^2};\quad \widehat{w}:= \Pi_{v}\widehat{A}|_{v^\perp};\quad \widehat{B}:=\Pi_{v^\perp}\widehat A|_{v^\perp},
$$
respectively.
\begin{corollary} For fixed, $v$, $\lambda$, and $w$, we get
$$
 \Esp_{B\sim\mathcal{G}(m-1,\sigma^2)_{\widehat B}}(\|(B-\lambda I_{n-1})^{-1}\|^2\,|\det(B-\lambda I_{n-1})|^2)\leq \frac{n-1}{\sigma^2}\Esp_{B\sim\mathcal{G}(m-1,\sigma^2)_{\widehat B}}(|\det(B-\lambda I_{n-1})|^2).
 $$
\end{corollary}
\begin{proof}
 Note that  
\[
\Esp_{B\sim\mathcal{G}(m-1,\sigma^2)_{\widehat B}}(\|(B-\lambda I_{n-1})^{-1}\|^2\,|\det(B-\lambda I_{n-1})|^2)= \Esp_{C\sim\mathcal{G}(m-1,\sigma^2)_{\widehat C}}(\|C^{-1}\|^2\,|\det C|^2),
\]
 where $\widehat C=\widehat B-\lambda I_{n-1}$. Now the proof follows immediately from Proposition \ref{prop:smooth}
\end{proof}
Now we prove the bound on the condition number. Note that
\[
\frac{\mu_F(A)^2}{\|A\|_F^2}=\frac{\max(1,\|A\|_F^2\|A_{\lambda,v}^{-1}\|^2)}{\|A\|_F^2}\leq \frac{1+\|A\|_F^2\|A_{\lambda,v}^{-1}\|^2}{\|A\|_F^2},
\]
and
\[
 \Esp_{A\sim\G(m,\sigma^2)_{\widehat{A}}}\left(\sum_{\lambda,v:Av=\lambda v}\frac{1}{\|A\|_F^2}\right)\leq\frac{1}{\sigma^2}.
\]
Hence, 
\begin{align*}
 \Esp_{A\sim\G(m,\sigma^2)_{\widehat{A}}}\left(\frac{\widehat{\mu_F^2}(A)}{\|A\|^2}\right)&\leq\frac{1}{\sigma^2}+
\int_{v\in\mathbb{P}(\C^n)} \frac{e^{-\|\hat{y}_v\|^2}}{\pi^{n-1}}\Esp_\lambda\Esp_{w}\Esp_{B\sim\mathcal{G}(m-1,\sigma^2)_{\widehat B}}(\|(B-\lambda I_{n-1})^{-1}\|^2\,|\det(B-\lambda I_{n-1})|^2)\,dv\\
&\leq
\frac{1}{\sigma^2}+\int_{v\in\mathbb{P}(\C^n)}\frac{e^{-\|\hat{y}_v\|^2}}{\pi^{n-1}} \Esp_\lambda\Esp_{w}\Esp_{B\sim\mathcal{G}(m-1,\sigma^2)_{\widehat B}}(|\det(B-\lambda I_{n-1})|^2)\,dv\\
&\underset{\text{\eqref{eq:citame}, $\alpha=0$}}{=}\frac{1}{\sigma^2}+\frac{n-1}{\sigma^2}\Esp_{A\sim\G(m,\sigma^2)_{\widehat{A}}}(n)=\frac{(n-1)n+1}{\sigma^2}\leq\frac{n^2}{\sigma^2}.
\end{align*}

This proves the first part of Theorem \ref{thm:mu2average} (note that we divide by $n$ to get the average also on the eigenpairs). For the second part of the theorem,  let
\[
Q=\frac{1}{Vol(\S)}\int_{A\in\S}\frac{1}{n}\sum_{(\lambda,v):Av=\lambda v}\mu_F(A,\lambda,v)^2\,dA
\]
be the quantity we want to compute. From the first part of the theorem we have
\[
\frac{1}{\pi^{n^2}}\int_{A\in\mnc}\frac{1}{n}\sum_{(\lambda,v):Av=\lambda v}\frac{\mu_F\left(A,\lambda,v\right)^2}{\|A\|_F^2}e^{-\|A\|_F^2}\,dA\leq n.
\]
On the other hand,
\[
\frac{1}{\pi^{n^2}}\int_{A\in\mnc}\frac{1}{n}\sum_{(\lambda,v):Av=\lambda v}\frac{\mu_F\left(A,\lambda,v\right)^2}{\|A\|_F^2}e^{-\|A\|_F^2}\,dA=
\]
\[
\frac{1}{\pi^{n^2}}\int_0^\infty\frac{e^{-\rho^2}}{\rho^2}\int_{A:\|A\|_F=\rho}\frac{1}{n}\sum_{(\lambda,v):Av=\lambda v}\mu_F\left(A,\lambda,v\right)^2\,dA\,d\rho.
\]
Now, because $\mu_F(A,\lambda,v)$ is invariant under multiplication of $A$ by a nonzero complex number, denoting $\nu_\rho=Vol(A:\|A\|_F=\rho)$, we have
\[
\frac{1}{\nu_\rho}\int_{A:\|A\|_F=\rho}\frac{1}{n}\sum_{(\lambda,v):Av=\lambda v}\mu_F\left(A,\lambda,v\right)^2\,dA=Q,\quad 0<\rho<\infty.
\]
We have thus proved:
\[
\frac{Q}{\pi^{n^2}}\int_0^\infty\frac{\nu_\rho e^{-\rho^2}}{\rho^2}\,d\rho\leq n.
\]
Note now that
\[
\nu_\rho=\frac{2\pi^{n^2}}{\Gamma(n^2)}\rho^{2n^2-1}
\]
to conclude
\[
Q\leq  \frac{(n-1)\Gamma(n^2)}{2\int_0^\infty\rho^{2n^2-3}e^{-\rho^2}\,d\rho}=\frac{n\Gamma(n^2)}{\Gamma(n^2-1)}=n(n^2-1)\leq n^3,
\]
and the theorem follows.
\section{Proof of Theorem \ref{th:deterministic1}}
We follow along the lines of the proof of \cite[Theorem 18.14]{BurgisserCucker2013}. First note that the homotopy \eqref{eq:At} is unchanged if $A$ is multiplied by some real positive number. Thus, the expected value of the number of homotopy steps needed to follow \eqref{eq:At} is equal for every centrally symmetric density function of $A$. In particular, for every $T=\sqrt{2}\, n>0$ and for every eigenvalue, eigenvector pair $(\lambda_0,v_0)$ of $A_0$ we have
\[
\E_{A\sim \G(n)}\left(\mathcal{C}(A,A_0,\lambda_0,v_0)\right)= \E_{A\sim \G_T(n)}\left(\mathcal{C}(A,A_0,\lambda_0,v_0)\right),
\]
where $\G_T(n)$ is the truncated Gaussian distribution described in Section \ref{sec:rv}. From Lemma \ref{lem:daigual} we then have that
\[
\E_{A\sim \G_T(n)}\left(\sum_{\lambda_0,v_0:A_0v_0=\lambda_0v_0}\mathcal{C}(A,A_0,\lambda_0,v_0)\right)\leq 
\]
\[
\E_{A\sim \G_T(n)}\left( c\|A_0\|_F\,\|A\|_F\int_0^1\frac{\sum_{\lambda_0,v_0:A_0v_0=\lambda_0v_0}\mu(A_t,\lambda_t,v_t)^2}{\|A_t\|_F^2}\,dt\right),
\]
where $A_t=(1-t)A_0+tA$ and $\lambda_t,v_t$ are defined by continuation. Now, note that the sum covers all the eigenvalue, eigenvector pairs of $A_t$ and thus we have:
\[
\E_{A\sim \G(n)}\left(\mathcal{C}(A,A_0,\lambda_0,v_0)\right)\leq  cT\|A_0\|_F\E_{A\sim \G_T(n)}\left( \int_0^1\frac{\hat{\mu}(A_t)^2}{\|A_t\|_F^2}\,dt\right)\underset{\text{\eqref{eq:truncated}}}{\leq}
\]
\[
2cT\|A_0\|_F\E_{A\sim \G(n)}\left( \int_0^1\frac{\hat{\mu}(A_t)^2}{\|A_t\|_F^2}\,dt\right).
\]
Let 
\[
t_0=\frac{\|A_0\|_F}{200\mu(A_0)^2(T+\|A_0\|_F)},
\]
and note that from Lemma \ref{lem:loclip}, using $\|A_0-A_t\|=t\|A-A_0\|\leq t(T+\|A_0\|_F)$ we have
\[
t\leq t_0\Rightarrow \|A_0-A_t\|\leq \frac{\|A_0\|_F}{200\mu(A_0)^2}\Rightarrow \frac{\hat{\mu}(A_t,\lambda_t,v_t)^2}{\|A_t\|_F^2}\leq  \frac{5n{\mu}(A_0)^2}{\|A_0\|_F^2}.
\]
Dividing the interval of integration $(0,1)$ into two pieces $(0,t_0)$ and $(t_0,1)$ We thus have
\[
\E_{A\sim \G(n)}\left(\mathcal{C}(A,A_0,\lambda_0,v_0)\right)\leq 2cT\|A_0\|_F\left(\frac{5n{\mu}(A_0)^2t_0}{\|A_0\|_F^2}+\E_{A\sim \G(n)}\left(\int_{t_0}^1\frac{\hat{\mu}(A_t)^2}{\|A_t\|_F^2}\,dt\right)\right).
\]
In order to bound the last term in the previous expression, we interchange the order of integration:
\[
\E_{A\sim \G(n)}\left(\int_{t_0}^1\frac{\hat{\mu}(A_t)^2}{\|A_t\|_F^2}\,dt\right)=\int_{t_0}^1\E_{A\sim \G(n)}\left(\frac{\hat{\mu}(A_t)^2}{\|A_t\|_F^2}\right)\,dt.
\]
Now, for fixed $t$, if $A\sim G(n)$ then $A_t=(1-t)A_0+tA$ satisfies $A_t\sim \G(n,t^2)_{(1-t)A_0}$ and from Theorem \ref{thm:mu2average} we have
\[
\E_{A\sim \G(n)}\left(\frac{\hat{\mu}(A_t)^2}{\|A_t\|_F^2}\right)\leq\frac{n^2}{t^2},
\]
which implies
\[
\E_{A\sim \G(n)}\left(\int_{t_0}^1\frac{\hat{\mu}(A_t)^2}{\|A_t\|_F^2}\,dt\right)\leq \int_{t_0}^1\frac{n^2}{t^2}\,dt\leq\frac{n^2}{t_0}.
\]
We have thus proved:
\[
\E_{A\sim \G(n)}\left(\mathcal{C}(A,A_0,\lambda_0,v_0)\right)\leq 2cT\|A_0\|_F\left(\frac{5n{\mu}(A_0)^2t_0}{\|A_0\|_F^2}+\frac{n^2}{t_0}\right)\leq
\]
\[
 K_1n+K_2n^3(\sqrt{2}\,n+\|A_0\|_F)\mu_F(A_0)^2,
\]
for some constants $K_1,K_2$. The same bound is valid for whatever scalar multiple of $A_0$ so we can consider $\|A_0\|_F$ arbitrarily small, and the bound claimed in the theorem holds.

Now, for the diagonal choice of $A_0$ in the theorem, recall some basic properties of the hexagonal lattice: say that we take hexagons with side of length $1$, so the distances between two centers is at least $\sqrt{3}$. The first hexagon $H_1$ is centered at $0$ and $\lambda_1=0$, then there are $6$ hexagons $H_2,\ldots,H_7$ whose centers $\lambda_2,\ldots,\lambda_7$ form another hexagon of side $\sqrt{3}$, then $12$ more hexagons $H_8,\ldots,H_{19}$ whose centers $\lambda_8,\ldots,\lambda_{19}$ form another hexagon of side $2\sqrt{3}$ and so on. Adding it up, if $k$ is the first natural number such that $n\leq 1+3k(k+1)$, i.e.,
\[
1+3k(k-1) \leq n\leq 1+3k(k+1),\quad\text{that is}\quad |n-(1+3k^2)|\leq 6k,
\]
we have that $\lambda_1,\ldots,\lambda_n$ are contained in an hexagon centered at $0$ of side $k\sqrt{3}$, and the associated side $1$ hexagons $H_1,\ldots, H_n$ are contained in a circle of radius $k\sqrt{3}+\sqrt{3}/2$. Moreover, note that for every $i$, $1\leq i\leq n$, denoting by $\nu=3\sqrt{3}/2$ the area of an hexagon of side $1$, we have
\[
|\lambda_i|^2=\frac{1}{\nu}\int_{x\in H_i}|\lambda_i|^2\,dx\leq\frac{1}{\nu}\int_{x\in H_i}(|x|+1)^2\,dx.
\]
We thus have
\[
\|A_0\|_F^2=\sum_{i=1}^n|\lambda_i|^2\leq \frac{1}{\nu}\sum_{i=1}^n\int_{x\in H_i}(|x|+1)^2\,dx\leq \frac{1}{\nu}\sum_{i=1}^n\int_{|x|\leq k\sqrt{3}+\sqrt{3}/2}(|x|+1)^2\,dx\]
\[
\leq\sqrt{3}\,\pi\,k^4+o(k^4)=\frac{\sqrt{3}\pi}{9}n^2+o(n).
\]
We conclude that
\[
\mu(A_0)^2\leq \left(\frac{\sqrt{3}\pi}{9}n^2+o(n)\right)\frac{1}{3}=\frac{\sqrt{3}\pi}{27}n^2+o(n),
\]
and the second claim of the theorem follows.

%

\section{Integration in the linear solution variety}

It will be useful to consider a geometrical scheme similar to that of Section \ref{sec:geometry} for the case of solving linear systems: we consider
\[
\Vlin=\{(M,v)\in\mnn\times\pc: \,Mv=0\},
\]
The linear solution variety $\Vlin$ is a $n(n-1)$--dimensional smooth submanifold of $\mnn\times\pc$, and again it inherits the Riemannian structure of the ambient space.

The linear solution variety is equipped with two projections
\begin{equation}\label{eq:pilin}
\begin{matrix}\pil:&\Vlin&\rightarrow&\mnn\\&(M,v)&\mapsto&M\end{matrix}\qquad  \begin{matrix}\pil_2:&\Vlin&\rightarrow&\pc\\&(M,v)&\mapsto&v\end{matrix}
\end{equation}
For $A\in\mnn$, $(\pil)^{-1}(A)$ is a copy of the kernel of $A$ in $\pc$, and for $v\in\pc$, $(\pil_2)^{-1}(v)$ is a copy of the linear subspace of $\mnn$ consisting of the matrices $A$ such that $Av=0$.

An argument similar to that of Section \ref{sec:interationV} can be used for integrating functions in $\Vlin$ using the projections in \eqref{eq:pilin}. In this case, it is well--known (see for example \cite[Lemma 3, p.242]{BlCuShSm98} that
\[
\frac{NJ(\pil)(M,v)}{NJ(\pil_2)(M,v)}=|\det(MM^*)|.
\]
We thus have
\begin{proposition}\label{prop:trick2}
Let $\phil:\Vlin\to[0,\infty]$ be a measurable unitarily invariant function in the sense that $\phil(M,v)=\phil(MU^*,Uv)$ for any unitary matrix $U\in\CU_n$. Then,
$$
\E_{M\sim\G((n-1)\times n)}(\phil(M,Ker(M))=\frac{1}{\Gamma{(n)}}
\E_{B\sim\G(n-1)}(\phil((0\;B),e_1)|\det(B)|^2).
$$
\end{proposition}
An immediate, self--interesting corollary follows:
\begin{corollary}\label{cor:exactpinvM}
\[
\E_{M\in \G((n-1)\times n)}(\|M^\dagger\|_F^2)=n-1.
\]
\end{corollary}
\begin{proof}
From Proposition \ref{prop:trick2} with $\phil(M,\zeta)=\|M^\dagger\|_F^2$, we have
\[
\E_{M\in \G((n-1)\times n)}(\|M^\dagger\|_F^2)=\frac{1}{\Gamma{(n)}}
\E_{B\sim\G((n-1)\times (n-1))}(\|B^{-1}\|_F^2|\det(B)|^2)\underset{\text{Prop. \ref{prop:smooth}}}{=}n-1.
\]
\end{proof}
Assume now that we are given any measurable nonnegative function $\alpha:\C^{(n-1)\times(n-1)}\to[0,\infty]$. We can produce a unitarily invariant function defined in $\Vlin$ by defining
\[
\phil(M,v)=\mathbb{E}_{Q\in\CS^{v}_{(n-1)\times n}}(\alpha(MQ^*)),\quad (M,v)\in\Vlin,
\]
where $\CS^{v}_{(n-1)\times n}$ is the set of $(n-1)\times n$ complex matrices such that $QQ^*=I_{n-1}$ and $Qv=0$. It is a simple exercise to check that $\phil$ is unitarily invariant.
Applying Proposition \ref{prop:trick2} to $\phil$ then yields:
$$
\E_{M}(\mathbb{E}_{Q\in\CS^{Ker(M)}_{(n-1)\times n}}(\alpha(MQ^*)))=\frac{1}{\Gamma{(n)}}
\E_{B}(\mathbb{E}_{U\in\CU_{n-1}}(\alpha(BU)|\det(B)|^2)).
$$
For any given full--rank $M\in\C^{(n-1)\times n}$, let $Q_M$ be the $Q$ factor in some reduced $QR$ decomposition of $M^*$. Note that
\[
\CS^{Ker(M)}_{(n-1)\times n}=\{(Q_MU)^*:U\in\CU_{n-1}\}.
\]
which defines an isometry bewteen $\CU_{n-1}$ and $\CS^{Ker(M)}_{(n-1)\times n}$. We have then proved:
$$
\E_{M}(\mathbb{E}_{U\in\CU_{n-1}}( \alpha(MQ_MU)))=\frac{1}{\Gamma{(n)}}
\E_{B}(\mathbb{E}_{U\in\CU_{n-1}}(\alpha(BU)|\det(B)|^2)).
$$
Finally, using Fubini's theorem we can interchange the integration order in the right--hand term, and then note that the isometry $B\mapsto BU$ preserves the value of the integral inside. We thus have:
\begin{corollary}\label{cor:trick2}
Let $\alpha:\C^{(n-1)\times(n-1)}\to[0,\infty]$ be an a.e. continuous function. Then,
$$
\E_{M\sim\G((n-1)\times n)}(\mathbb{E}_{U\in\CU_{n-1}}( \alpha(MQ_MU)))=\frac{1}{\Gamma{(n)}}
\E_{B\sim\G((n-1)\times (n-1))}(\alpha(B)|\det(B)|^2),
$$
where $M\rightarrow Q_M$ is any almost everywhere continous function sending $M$ to the $Q$ factor in the reduced $QR$--decomposition of $M^*$.
\end{corollary}
Note that the a.e. continuity of $\alpha$ and $M\mapsto Q_M$ makes the composite mapping $(M,U)\mapsto\alpha(MQ_MU)$ a.e. continuous and hence measurable, so the integrals in Corollary \ref{cor:trick2} make sense.

\section{Proof of $(3)$ in Theorem \ref{th:mainrandom}}
We are now prepared for proving \eqref{eq:wish2}. From the definition and Fubini's theorem, for any a.e. continuous nonnegative function $\phi$ defined on $\V$, the expected value $\mathbb{E}_{w\in \Omega_n}\left(\phi(\varphi_n(w))\right)$ equals:
\[
C_n\mathbb{E}_{M,U}\left(\E_{z,w}\left(\phi\left(\begin{pmatrix}
z&w^*\\0&MQ_MU
\end{pmatrix},z,e_1
\right)\;  \chi_{n|z|\,\|M^\dagger\|_F\leq1}\right)\right)=C_n\mathbb{E}_{M,U}\left(\alpha(MQ_MU)\right),
\]
where $z\sim\G(1,(2n^3)^{-1})$, $M\sim \G((n-1)\times n)$, $U\in\CU_{n-1}$, $w\in\G(1\times(n-1))$ and $\alpha:\C^{(n-1)\times(n-1)}\rightarrow[0,\infty]$ is defined by
\[
\alpha(B)=\E_{z,w}\left(\phi\left(\begin{pmatrix}
z&w^*\\0&B
\end{pmatrix},z,e_1
\right)\;  \chi_{n|z|\,\|B^{-1}\|_F\leq1}\right).
\]
We are then under the hypotheses of Corollary \ref{cor:trick2}. We then conclude:
\[
\mathbb{E}_{w\in \Omega_n}\left(\phi(\varphi_n(w))\right)=\frac{C_n}{\Gamma{(n)}}
\E_{B\sim\G(n-1)}(\alpha(B)|\det(B)|^2)=
\]
\begin{equation}\label{eq:ll}
\frac{C_n}{\Gamma{(n)}}
\E_{B}\left(\E_{z,w}\left(\phi\left(\begin{pmatrix}
z&w^*\\0&B
\end{pmatrix},z,e_1
\right)\;  \chi_{n|z|\,\|B^{-1}\|_F\leq1}\right)|\det(B)|^2\right).
\end{equation}
Our goal is to transform this integral into the right--hand term of the equality in Proposition \ref{prop:trick}. Let
\[
S_n=\sup_{z\in\C,B\in\C^{(n-1)\times (n-1)}:n|z|\,\|B^{-1}\|_F\leq1}\frac{|\det (B)|^2}{|\det(B-zI_{n-1})|^2}.
\]
Then, multiplying and dividing in \eqref{eq:ll} the integrand by $|\det(B-z I_{n-1})|^2$ we get:
\[
\mathbb{E}_{w\in \Omega_n}\left(\phi(\varphi_n(w))\right)\leq \frac{C_nS_n}{\Gamma{(n)}}\E_{B}\left({\E_{z,w}}\left(\phi\left(\begin{pmatrix}
z&w^*\\0&B
\end{pmatrix},z,e_1
\right)\right)|\det(B-z I_{n-1})|^2\right).
\]
Note that $z$ in the last expression is distributed as $z\sim\G(1,(2n^{3})^{-1})$. Noting that
\[
e^{-2n^3|z|^2}\leq e^{-|z|^2},
\]
we conclude
\[
\mathbb{E}_{w\in \Omega_n}\left(\phi(\varphi_n(w))\right)\leq \frac{2n^3C_nS_n }{\Gamma{(n)}}\E_{B,\lambda,w}\left(\phi\left(\begin{pmatrix}
\lambda&w^*\\0&B
\end{pmatrix},\lambda,e_1
\right)|\det(B-\lambda I_{n-1})|^2\right),
\]
where $\lambda$ is chosen in $N_\C(0,1)$. From Proposition \ref{prop:trick} we then have proved that
\[
\mathbb{E}_{w\in \Omega_n}\left(\phi(\varphi_n(w))\right)\leq 2C_nS_n n^3\E_{A\sim\G(n)}(\widehat\phi(A)).
\]
Claim $(3)$ in Theorem \ref{th:mainrandom} then follows from estimating $C_n$ and $S_n$, which we do in the Lemma \ref{lem:SnCn} below. This finishes the proof of Theorem \ref{th:mainrandom}.
\begin{lemma}\label{lem:SnCn}
In the notations above,
\[
S_n\leq 2e,\quad C_n\leq 2.
\]
\end{lemma}
\begin{proof}
We first prove the inequality for $S_n$. Let $t=|z|$ and note that the hypotheses $nt\|B^{-1}\|_F\leq1$ implies that the operator norm of $B^{-1}$ is at most $(nt)^{-1}$, that is all the singular values of $B$ are greater than or equal to $nt$, and consequently all the eigenvalues $\lambda_1,\ldots,\lambda_{n-1}$ of $B$ have modulus at least $nt$. We then have
\[
\frac{|\det (B)|^2}{|\det(B-t I_{n-1})|^2}=\frac{|\lambda_1\cdots\lambda_{n-1}|^2}{|(\lambda_1-t)\cdots(\lambda_{n-1}-t)|^2}=\prod_{i=1}^{n-1}\left|1+\frac{t}{\lambda_i-t}\right|\leq \prod_{i=1}^{n-1}\left(1+\frac{t}{|\lambda_i|-t}\right)\leq
\]
\[
\left(1+\frac{t}{nt-t}\right)^n=\left(1+\frac{1}{n-1}\right)^n\leq 2\left(1+\frac{1}{n-1}\right)^{n-1}\leq2e.
\]
We now prove the bound on $C_n$. from the definition we have
\[
C_n^{-1}=\P_{z\sim\G(1,(2n^3)^{-1}),M\sim\G(n-1,n)}\left(\|M^\dagger\|\leq \frac{1}{nt}\right)\underset{\text{\eqref{eq:prob}}}{\geq}1-\frac{1}{2}=\frac{1}{2}.
\]
so we have $C_n\leq 2$ as claimed.
\end{proof}

\section{Proof of Theorem \ref{th:randomhomotopy}}\label{sec:proofofrandomisrandom}
From Proposition \ref{prop:stepsold}, the expected number of homotopy steps is at most
\[
K\leq \frac{c}{\pi^{2n^2}}\int_{A_0\in\mnc}\frac{1}{n}\sum_{(\lambda_0,v_0):A_0v_0=\lambda_0 v_0}\int_{A\in\mnc}\int_0^{d_\S(A_0,A)}\mu_F(A_t,\lambda_t,v_t)^2\,dt\; e^{-\|A\|^2_F}e^{-\|A_0\|_F^2}\,dA\,dA_0,
\]
where $A_t$ is as in \eqref{eq:At} and $\lambda_t,v_t$ are defined by continuation with $\lambda_0=\lambda,\;v_0=v$. Note that for a.e. $B$ and a.e. $A$ the eigenpairs of $A_t,0\leq t\leq 1$ are in one to one correspondence with the eigenpairs of $A$ by continuation, hence we can simply write
\[
K\leq \frac{c}{\pi^{2n^2}}\int_{A_0\in\mnc}\int_{A\in\mnc}\int_0^{d_\S(A_0,A)}\frac{1}{n}\sum_{(\lambda,v):A_tv=\lambda v}\mu_F(A_t,\lambda,v)^2\,dt\; e^{-\|A\|^2_F}e^{-\|A_0\|_F^2}\,dA\,dA_0,
\]
Now, note that only $e^{-\|A\|^2_F}e^{-\|A_0\|_F^2}$ varies when multiplying $A$ or $A_0$ by a constant. Thus, using the notation $\S=\S(\mnc)$ for the unit sphere in $\mnc$, we can rewrite
\[
K\leq \frac{c}{\pi^{2n^2}}\int_{r,s\in[0,\infty)}\int_{\|A\|_F=r,\|A_0\|_F=s}\int_0^{d_\S(A_0,A)}\frac{1}{n}\sum_{(\lambda,v):A_tv=\lambda v}\mu_F(A_t,\lambda,v)^2\,e^{-r^2-s^2}\,dt\,d(A,A_0)\,d(r,s)=
\]
\[
\frac{c}{\pi^{2n^2}}\int_{r,s\in[0,\infty)}(rs)^{2n^2-1}\,e^{-r^2-s^2}\,d(r,s)\int_{A,A_0\in\S}\int_0^{d_\S(A,A_0)}\frac{1}{n}\sum_{(\lambda,v):A_tv=\lambda v}\mu_F(A_t,\lambda,v)^2\,dt\,d(A,A_0)\underset{\text{Lemma \ref{lem:auxx}}}{=}
\]
\[
c\mathbb{E}_{A,A_0\in\S}\left(\int_0^{d_\S(A_0,A)}\frac{1}{n}\sum_{(\lambda,v):A_tv=\lambda v}\mu_F(A_t,\lambda,v)^2\,dt\right)\underset{\text{Lemma \ref{lem:rt}}}{=}
\]
\[
\frac{\pi c}{2}\,\mathbb{E}_{A\in\S}\left(\frac{1}{n}\sum_{(\lambda,v):Av=\lambda v}\mu_F(A,\lambda,v)^2\right)\underset{\text{Th. \ref{thm:mu2average}}}{\leq}\frac{\pi c}{2} n^3.
\]
The theorem follows. The total number of arithmetic operations is obtained by multiplying the number of homotopy steps by the cost of one homotopy step, that is $O(n^3)$.

\begin{lemma}\label{lem:auxx}
\[
\int_{r\in[0,\infty)}r^{2n^2-1}e^{-r^2}\,dr=\frac{\pi^{n^2}}{Vol(\S)}.
\]
\end{lemma}
\begin{proof}
\[
\int_{r\in[0,\infty)}r^{2n^2-1}e^{-r^2}\,dr=\frac{1}{Vol(\S)}\int_{r\in[0,\infty)}Vol(A\in\mnc:\|A\|_F=r)e^{-r^2}\,dr=
\]
\[
\frac{1}{Vol(\S)}\int_{A\in\mnc}e^{-\|A\|_F^2}\,dA=\frac{\pi^{n^2}}{Vol(\S)}.
\]
\end{proof}
\begin{lemma}\label{lem:rt}
For every measurable nonnegative mapping $\psi:\S\rightarrow[0,\infty]$, we have
\[
\mathbb{E}_{A,A_0\in\S}\left(\int_0^{d_\S(A,A_0)}\phi(A_t)\,dt\right)=\frac{\pi}{2} \mathbb{E}_{A\in\S}(\phi(A))
\]
where $A_t$ is given by \eqref{eq:At}.
\end{lemma}
\begin{proof}
Note that the mapping
\[
\phi\mapsto \mathbb{E}_{A,A_0\in\S}\left(\int_0^{d_\S(A,A_0)}\phi(A_t)\,dt\right)
\]
defines a measure in $\S$. Moreover, this measure is invariant under the simmetry group of the sphere, and is thus a constant multiple of the standard measure in $\S$. Namely, there is a constant $C=C(n)$ such that
\[
\mathbb{E}_{A,A_0\in\S}\left(\int_0^{d_\S(A,A_0)}\phi(A_t)\,dt\right)= C\mathbb{E}_{A\in\S}(\phi(A))
\]
To compute $C$, let $\phi\equiv 1$. We then get
\[
\mathbb{E}_{A,A_0\in\S}\left(d_\S(A,A_0)\right)= C.
\]
Note that the change of variables $A_0\mapsto -A_0$ does not change the expected value in this last formula. Moreover, $d_S(A,A_0)+d_S(A,-A_0)=\pi$ for all $A,A_0\in\S$. Thus,
\[
2C=\mathbb{E}_{A,A_0\in\S}\left(d_\S(A,A_0)\right)+\mathbb{E}_{A,A_0\in\S}\left(d_\S(A,-A_0)\right)=\mathbb{E}_{A,A_0\in\S}\left(d_\S(A,A_0)+d_\S(A,-A_0)\right)=\pi,
\]
proving that $C=\pi/2$. The lemma follows.
\end{proof}

\appendix
\section{The quotient of Normal Jacobians}\label{sec:nj}
The tangent space $T_{(A,\lambda,v)}\mathcal{V}$ to $\V$ at $(A,\lambda,v)$ is the set of triples
$$
(\dot{A},\dot{\lambda},\dot{v})\in \mnc\times \C\times \C^n,
$$
satisfying
\begin{equation}\label{tanspaceV}
 (\dot\lambda I_n-\dot A) v +(\lambda I_n-A)\dot v=0, \quad  \pes{\dot v}{v}=0.
\end{equation}

By the implicit function theorem, for every $(A,\lambda,v)\in\W$, one can (locally) define a map $\mathcal{S}_{(A,\lambda,v)}:=\pi_2\circ\pi^{-1}$ from a neighborhood of $A\in\mnc$ onto a neighborhood of $v\in\pc$. Its derivative $D\mathcal{S}_{(A,\lambda,v)}:\mnc\to T_{v}\pc$ is given by:
\begin{equation}\label{eq:DS}
D\mathcal{S}_{(A,\lambda,v)}=D\pi_2(A,\lambda,v)\circ (D\pi(A,\lambda,v))^{-1}.
\end{equation}
From (\ref{tanspaceV}), it is easily seen that
$$
D\mathcal{S}_{(A,\lambda,v)}\dot A = (A_{\lambda,v})^{-1}\Pi_{v^\perp}\dot Av.
$$

Let $H_{\mathcal{S}_{(A,\lambda,v)}}(A)$ be the horizontal space of $\mnc$  associated to $\mathcal{S}_{(A,\lambda,v)}$. (Recall that the horizontal space of a linear applications is the Hermitian complement of its kernel.)  Then, the normal jacobian $NJ_{\mathcal{S}{(A,\lambda,v)}}(A)$ is given by the horizontal derivatavie, i.e.,  
$$
NJ_{\mathcal{S}_{(A,\lambda,v)}}=|\det(D\mathcal{S}_{(A,\lambda,v)})|_{H_{\mathcal{S}_{(A,\lambda,v)}}}|^2.
$$
A straightforward computation yields the following lemma.
\begin{lemma}\label{lem:HorizontalS}
For every $(A,\lambda,v)\in\W$, one has,
\begin{enumerate}
 \item  $H_{\mathcal{S}_{(A,\lambda,v)}}(A)=\{wv^*:\,w\in v^\perp\}$;
 \item $NJ_{\mathcal{S}_{(A,\lambda,v)}}=|\det({A_{\lambda,v}}^{-1})|^2.$
\end{enumerate}
\qed
\end{lemma}

\begin{remark}
 Since the linear operator $A_{\lambda,v}:v^\perp\to v^\perp$ is invariant under the action of the unitary group, i.e., $(UAU^{-1})_{\lambda,Uv}=UA_{\lambda,v}U^{-1}$, we conclude that the real-valued map $|\det((A_{\lambda,v})^{-1})|^2$ is unitarily invariant in $\W$.
\end{remark}
 \begin{lemma}\label{lem:quotient-NJ}
 Let $(A,\lambda,v)\in\W$. Then, 
 $$
 \frac{NJ_{\pi }}{NJ_{\pi_2}}(A,\lambda,v)=\frac{1}{2}|\det({A_{\lambda,v}})|^2.
 $$
 In particular $ \frac{NJ_{\pi }}{NJ_{\pi_2}}(A,\lambda,v)$ is unitarily invariant.
\end{lemma}
\begin{proof}
Fix $(A,\lambda,v)\in\W$. Abusing the notation, we will drop the base point $(A,\lambda,v)$ during the proof.  The derivatives of the canonical projections are given by
 $$
 D\pi(\dot A,\dot\lambda,\dot v)=\dot A,\quad\mbox{and}\quad D\pi_2(\dot A,\dot\lambda,\dot v)=\dot v,
 $$
 where $(\dot A,\dot\lambda,\dot v) \in\mnc\times\C\times\C^n$ satisfies (\ref{tanspaceV}). 
 Let $H_{\pi_2}\subset T_{(A,\lambda,v)}\V$ and $H_{\mathcal{S}}\subset\mnc$ be the horizontal spaces associated to $\pi_2$ and $\mathcal{S}=\mathcal{S}_{(A,\lambda,v)}$. 
We divide the rest of the proof into some easy claims:

\begin{enumerate}
 \item[(i)] $D\pi({\ker D\pi_2})= \ker(D\mathcal{S})$: since $D\pi$ is an isomorphism, this follows from (\ref{eq:DS}).
\item[(ii)]   $D\pi({H_{\pi_2}})= H_{\mathcal{S}}$: from (1), by a dimension argument, it is enough to prove $ H_{\mathcal{S}}\subset D\pi({H_{\pi_2}})$.  From Lemma~\ref{lem:HorizontalS}, and (\ref{tanspaceV}), we obtain that 
$$
(D\pi)^{-1}(H_{\mathcal{S}})=\{(wv^*,0,\|v\|^2(A_{\lambda,v})^{-1}w),\quad w\in v^{\perp}\}.$$
Furthermore, from (\ref{tanspaceV}), the kernel of $D\pi_2$ is given by the set of triples $(\dot A,\dot \lambda,0)\in\mnc\times\C$ such that $\dot Av=\dot\lambda v$. Thus $ (D\pi)^{-1}(H_{\mathcal{S}})\subset H_{\pi_2}$.
\item[(iii)] $NJ_{\pi}=|\det(D\pi|_{\ker D\pi_2})|^2\cdot|\det(D\pi|_{H_{\pi_2}})|^2$: this follows immediately from (i) and (ii).
\item[(iv)] $|\det(D\pi|_{\ker(D\pi_2)})|^2=1/2$: let $R_v=\frac{vv^*}{\|v\|^2}\in\mnc$. Then $(R_v,1,0)\in \ker (D\pi_2)$. Furthermore, the Hermitian complement of $(R_v,1,0)$ relative to $\ker(D\pi_2)$ is given by the set $\mathbb{A}_v\times\{0\}\times\{0\}\subset T_{(A,\lambda,v)}\V$, where
$$
\mathbb{A}_v=\{\dot B\in\mnc: \,\dot Bv=0\}.$$
Since $D\pi|_{\mathbb{A}_v\times\{0\}\times\{0\}}:\mathbb{A}_v\times\{0\}\times\{0\}\to\mathbb{A}_v$, is an isometry, and $D\pi(R_v,1,0) \in{\mathbb{A}_v}^{\perp}$, we conclude that 
$$
|\det(D\pi|_{\ker(D\pi_2)})|^2=\frac{\left\|D\pi(R_v,1,0) \right\|^2}{\|R_v\|_F^2+1}=\frac{\|R_v\|_F^2}{\|R_v\|_F^2+1}=\frac12.
$$

 \end{enumerate}

\vspace{3pt}

From Claim I and (\ref{eq:DS}) we get
$$
D\mathcal{S}|_{H_\mathcal{S}}=D\pi_2|_{H_{\pi_2}}\circ (D\pi|_{H_{\pi_2}})^{-1},
$$
and then from (iv) we conclude
$$
NJ_{\mathcal{S}}=\frac{|\det(D\pi_2|_{H_{\pi_2}} )|^2}{|\det(D\pi|_{H_{\pi_2}} )|^2}=\frac{1}{2}\frac{NJ_{\pi_2}}{NJ_{\pi}}.
$$
From this last equality and Lemma \ref{lem:HorizontalS} we have
\[
\frac{NJ_{\pi}}{NJ_{\pi_2}}=\frac{1}{2NJ_{\mathcal{S}}}=\frac{1}{2}|\det A_{\lambda,v}|^2,
\]
as wanted.
\end{proof}
\section{A explicit description of the algorithm}\label{appendix2}
We propose the following algorithm on input $(A_0,\lambda_*,v_*)$ and $A$. The precondition
 \[
d_{\P^2}((A,\lambda_*,v_*),(A,\lambda_0,v_0))\leq \frac{c_0}{4\mu(A,\lambda_0,v_0)}, \quad c_0=0.0739.
\]
which from Proposition  \ref{prop:Armentano2014} guarantees that $(\lambda_*,v_*)$ is an approximate eigenvalue, eigenvector pair of $A$ with associated exact pair some $(\lambda_0,v_0)$, is assumed. Also, recall that $B_t$ is defined by \eqref{eq:At}.
\begin{enumerate}
\item Let $t=0$, $c_0=0.0739$. While $t<a$ do:
\begin{itemize}
\item Let $(A_{now},\lambda_{now},v_{now})=(B_t,\lambda_*,v_*)$.
\item Compute $b>0$ satisfying
\[
\frac{C_\epsilon}{6\sqrt{2}(1+\epsilon)\mu(A_{now},\lambda_{now},v_{now})^2}\leq b\leq \frac{C_\epsilon}{2\sqrt{2}(1+\epsilon)\mu(A_{now},\lambda_{now},v_{now})^2},
\]
where $\epsilon=1/16$.
\item Let $t=t+b$, $A_{next}=B_t$ or $A_a$ if $t+b>a$.
\item Let $(\lambda_*,v_*)=N_{A_{next}}(\lambda_{now},v_{now})$.
\end{itemize}
\item Output $(\lambda_*,v_*)$.
\end{enumerate}

The main result in this section is:
\begin{theorem}\label{th:algorithm}
The algorithm above outputs an approximate eigenvalue, eigenvector pair of $A_a$ with associated exact pair the one conditnued by the homotopy in \eqref{eq:At}. Moreover, the total number of homotopy steps (i.e. the number of times the while loop is executed) is at most
\[
1000\int_0^a \mu(B_t,\lambda_t,v_t)^2\,dt.
\]
Moreover, the total number of arithmetic operations is $O(n^3)$ times this quantity.
\end{theorem}
We start the proof of Theorem \ref{th:algorithm} by recalling the following result:
\begin{proposition}\label{prop:variationmu0}\cite[Prop. 3.22]{Armentano2014}
Let $ \epsilon>0$ and let $(A,\lambda,v),(A',\lambda',v')\in\CW$ be such that $\mu(A,\lambda,v)<\infty$. Assume moreover that
\[
d_{\P^2}((A,\lambda,v),(A',\lambda',v'))\leq  \frac{C_\epsilon}{\mu(A,\lambda,v)},
\]
where 
\[
C_\epsilon=\frac{\arctan\frac{\epsilon}{\sqrt{2}+\alpha(1+\epsilon)}}{1+\epsilon},\quad \alpha=2\sqrt{2}(1+\sqrt{5}).
\]
Then,
\[
\frac{\mu(A,\lambda,v)}{1+\epsilon}\leq\mu(A',\lambda',v')\leq(1+\epsilon)\mu(A,\lambda,v)
\]

\end{proposition}
\begin{proposition}\label{prop:zetamoves}
Let $\epsilon>0$ and $A_0\not\in\Sigma$. Then, for every $A\in\C^{n\times n}$ such that the spherical distance $a=d_\S(A,A_0)$ from $A_0$ to $A$ satisfies 
\begin{equation}\label{eq:a}
a\leq \frac{C_\epsilon}{2\sqrt{2}(1+\epsilon)\mu(A_0,\lambda_0,v_0)^2},
\end{equation}
 we have
\[
d_{\P^2}((A,\lambda,v),(A_0,\lambda_0,v_0))\leq  \frac{C_\epsilon}{\mu(A_0,\lambda_0,v_0)}.
\]
where $\lambda,v$ is the eigenvalue, eigenvector pair of $A$ obtained by the homotopy \eqref{eq:At} starting at $(A_0,\lambda_0,v_0)$ (in particular, such homotopy is well--defined). Moreover, for all $t\in[0,a]$ we also have:
\[
\frac{1}{1+\epsilon}\mu(A_0,\lambda_0,v_0)\leq \mu(A_t,\lambda_t,v_t)\leq (1+\epsilon)\mu(A_0,\lambda_0,v_0)
\]
\end{proposition}
\begin{proof}
Let $B_t$ be given by \eqref{eq:At}, so $B_0=A_0$, $B_a=A$. Let $(\lambda_t,v_t)$ be the eigenvalue, eigenvector pair continued from $(B_0,\lambda_0,v_0)$, which is by the inverse function theorem defined for all $t<t_0$ for some $t_0>0$. Let
\[
t_1=\sup\{t\in[0,t_0]:d_{\P^2}((A_s,\lambda_s,v_s),(A_0,\lambda_0,v_0))<C_\epsilon/\mu(A_0,\lambda_0,v_0),\;\forall\;s\in[0,t]\}>0.
\]
Now, let $t\leq t_1$, and note that from \cite[Prop. 3.14, ii)]{Armentano2014} we have
\[
d_{\P^2}((A_t,\lambda_t,v_t),(A_0,\lambda_0,v_0))=\int_0^t\frac{d}{ds}d_{\P^2}((A_s,\lambda_s,v_s),(A_0,\lambda_0,v_0))\,ds< 2\sqrt{2}\int_0^t\mu(A_s,\lambda_s,v_s)\,ds.
\]
From Proposition \ref{prop:variationmu0} we conclude:
\[
d_{\P^2}((A_t,\lambda_t,v_t),(A_0,\lambda_0,v_0))<2\sqrt{2}(1+\epsilon)\mu(A_0,\lambda_0,v_0)\,t.
\]
Now note that for $t\leq t_1$ from Proposition \ref{prop:variationmu0} the condition number is bounded above by $(1+\epsilon) \mu(A_0,\lambda_0,v_0)$ and thus the solution can be continued. Hence, $t_0\geq t_1$. Now, assume that
\begin{equation}\label{eq:iamfalse}
t_1<\frac{C_\epsilon}{2\sqrt{2}(1+\epsilon)\mu(A_0,\lambda_0,v_0)^2}.
\end{equation}
In that case, the argument above shows that
\[
d_{\P^2}((A_{t_1},\lambda_{t_1},v_{t_1}),(A_0,\lambda_0,v_0))< 2\sqrt{2}(1+\epsilon)\mu(A_0,\lambda_0,v_0)\,t_1\leq C_\epsilon/\mu(A_0,\lambda_0,v_0),
\]
and by continuity for some small enough $s>0$ we have $d_{\P^2}((A_s,\lambda_s,v_s),(A_0,\lambda_0,v_0))>C_\epsilon/\mu(A_0,\lambda_0,v_0)$. This contradicts the definition of $t_1$, so we conclude that \eqref{eq:iamfalse} is false. We have thus proved
\[
\frac{C_\epsilon}{2\sqrt{2}(1+\epsilon)\mu(A_0,\lambda_0,v_0)^2}\leq t_1\leq t_0,
\]
that is for every $t$ smaller than the left--hand term, the pair $(\lambda_t,v_t)$ is well--defined by continuation and $ (1+\epsilon)^{-1}\mu(A_0,\lambda_0,v_0)\leq \mu(A_t,\lambda_t,v_t)\leq (1+\epsilon)\mu(A_0,\lambda_0,v_0)$ as claimed.
\end{proof}
\begin{proposition}\label{prop:onestep}
Let $c_0=0.0739$ and let $(A,\lambda_*,v_*)\in\CM\times\C\times\pc$ satisfy
\[
d_{\P^2}((A,\lambda_*,v_*),(A,\lambda_0,v_0))\leq \frac{c_0}{4\mu(A,\lambda_0,v_0)},
\]
where $(A,\lambda_0,v_0)\in\CW$. Then, for every $A\in\C^{n\times n}$ such that the spherical distance $a=d_\S(A,A_0)$ from $A_0$ to $A$ satisfies  \eqref{eq:a}  with $\epsilon=1/6$, namely,
\begin{equation}\label{eq:a2}
a\leq \frac{0.003579}{\mu(A_0,\lambda_0,v_0)^2},
\end{equation}
we have
\[
d_{\P^2}((A,\lambda,v),(A,N_A(\lambda_*,v_*)))\leq  \frac{c_0}{4\mu(A,\lambda,v)}.
\]
where $\lambda,v$ is the eigenvalue, eigenvector pair of $A$ obtained by the homotopy \eqref{eq:At} starting at $(A_0,\lambda_0,v_0)$.
\end{proposition}
\begin{proof}
Let $\epsilon=1/6$ which implies $C_\epsilon<c_0/6$. From Proposition \ref{prop:zetamoves}, we have
\[
d_{\P^2}((A,\lambda_*,v_*),(A,\lambda,v))\leq d_{\P^2}((A,\lambda_*,v_*),(A,\lambda_0,v_0))+d_{\P^2}((A,\lambda_0,v_0),(A,\lambda,v))\leq
\]
\[
 \frac{c_0}{4\mu(A,\lambda_0,v_0)}+ \frac{C_\epsilon}{\mu(A,\lambda_0,v_0)}\leq \frac{3c_0}{4\mu(A,\lambda_0,v_0)}\underset{\text{Prop. \ref{prop:zetamoves}}}{\leq} \frac{5c_0}{12\mu(A,\lambda,v)}(1+1/6)\leq  \frac{c_0}{2\mu(A,\lambda,v)}.
\]
From Proposition \ref{prop:Armentano2014} we conclude that $(\lambda_0,v_0)$ is an approximate zero of $A$ with associated exact zero the continued pair $(\lambda,v)$. From the definition of approximate zero, the distance is halved after one iteration of Newton's method which gives the desired result.
\end{proof}
\subsection{Proof of Theorem \ref{th:algorithm}}
We prove that at every step of the algorithm we have
\begin{equation}\label{eq:mirame}
d_{\P^2}((A_{now},\lambda_{now},v_{now}),(A_{now},\lambda_t,v_t))\leq \frac{c_0}{4\mu(A_{now},\lambda_t,v_t)}.
\end{equation}
(note that$(A_{now},\lambda_t,v_t)\in\CV$). The proof goes by induction. The base case for the first loop is true by hypotheses. Now, given that \eqref{eq:mirame} holds at a certain point, let us check that at the next step of the loop it still holds. Note that
\[
d_\S(A_{now},A_{next})=b\leq  \frac{C_\epsilon}{2\sqrt{2}(1+\epsilon)\mu(A_{now},\lambda_{now},v_{now})^2},\quad \epsilon=\frac{1}{16}.
\]
Now we need to compare $\mu(A_{now},\lambda_{now},v_{now})$ and $\mu(A_{now},\lambda_t,v_t)$. Note that we are under the hypotheses of Proposition \ref{prop:changemu}, with $\epsilon=5c_0/4<1/10$, so we have
\[
\frac{\mu(A_{now},\lambda_t,v_t)}{1-1/10}\geq \mu(A_{now},\lambda_{now},v_{now})\geq\frac{\mu(A_{now},\lambda_t,v_t)}{1+1/10},
\]
and hence
\[
b\leq  \frac{C_\epsilon}{2\sqrt{2}(1+\epsilon){\mu(A_{now},\lambda_t,v_t)^2}}(1+1/10)^2\leq  \frac{2C_\epsilon}{2\sqrt{2}(1+\epsilon){\mu(A_{now},\lambda_t,v_t)^2}}.
\]
By taking $\epsilon=1/16$ we guarantee that the constant term in the right hand side is at most $0.003579$. The induction step now follows from Proposition \ref{prop:onestep}. For the upper complexity bound, 
 from Proposition \ref{prop:zetamoves} and Proposition \ref{prop:changemu}, for every $s\in[t,t+b]$ we have
\[
\mu(A_s,\lambda_s,v_s) \geq \frac{ 16\mu(A_{now},\lambda_{t},v_{t})}{17}\geq\frac{ 4\mu(A_{now},\lambda_{now},v_{now})}{5}.
\]
We thus have at every step of the loop:
\[
\int_t^{t+b}\mu(A_s,\lambda_s,v_s)^2\,ds\geq \frac{ 4^2\,b\,\mu(A_{now},\lambda_{now},v_{now})^2}{5^2}\geq \frac{16C_\epsilon}{6\sqrt{2}(1+\epsilon)25}\geq\frac{3}{10^4}.
\]
We thus have by induction, after $k$ steps of the loop if the termination condition is not reached:
\[
\frac{3k}{10^4}\leq \int_0^{t+b}\mu(A_s,\lambda_s,v_s)^2\,ds\leq  \int_0^{a}\mu(A_s,\lambda_s,v_s)^2\,ds,
\]
and the claim on the upper bound of the number of steps follows. Note finally that each step requires some linear algebra operations such as matrix inversion which run in time $O(n^3)$, as well as the estimation up to a factor of $3$ of the squared operator norm of a matrix (to compute $\mu(A_{now},\lambda_{now},v_{now}$). Such a bound on the squared operator norm of any matrix $X$ can be easily computed by first transforming $XX^*$ to triangular Hessenberg form (which requires $O(n^3)$ operations) and then using the bound in \cite[Page 1]{Kahan1966} which computes the norm of a symmetric tridiagonal matrix up to a factor of $\sqrt{3}$.
\section{The variation of $\mu$ out of $\CW$}
\begin{lemma}\label{lem:moore}
Let $R,R'\in\C^{n\times n}$ where $R$ and $R'$ have rank $n-1$. Assume moreover that
\[
\|R-R'\|\leq \frac{\epsilon}{\|R^\dagger\|},\quad 0\leq\epsilon\leq 1/2.
\]
Then,
\[
\frac{1}{1+\epsilon}\|R^\dagger\|\leq \|R'^\dagger\|\leq \frac{\|R^\dagger\|}{1-\epsilon}.
\]
\end{lemma}
\begin{proof}
Let $\sigma$ ($\sigma'$) be the smallest nonzero singular value of $R$ ($R'$). Note that $\sigma-\|R-R'\|\leq \sigma'\leq \sigma+\|R-R'\|$ (see \cite[Cor. 8.6.2]{GoVa96}. We then have
\[
\|R^\dagger\|=\frac{1}{\sigma}=\frac{1}{\sigma'}\frac{\sigma'}{\sigma}\leq\|R'^\dagger\|\frac{\sigma+\|R-R'\|}{\sigma}\leq(1+\epsilon)\|R'^\dagger\|.
\]
The upper bound follows from a similar argument.
\end{proof}
\begin{proposition}\label{prop:changemu}
Let $(A,\lambda,v)\in\CW$, $(A,\lambda',v')\in\CM\times\C\times\pc$ be such that $\|A\|_F=1$ and
\begin{equation}\label{eq:mmimi}
d_{\P^2}((A,\lambda,v),(A,\lambda',v'))\leq \frac{\epsilon}{5\mu(A,\lambda,v)}
\end{equation}
Then,
\[
\frac{1}{1+\epsilon}\mu(A,\lambda,v)\leq \mu(A,\lambda',v')\leq \frac{\mu(A,\lambda,v)}{1-\epsilon}.
\]
\end{proposition}
\begin{proof}
Recall that
\[
\mu(A,\lambda,v)=\|((I- vv^*)(\lambda I-A))^\dagger\|,\quad \mu(A,\lambda',v')=\|((I- v'v'^*)(\lambda' I-A))^\dagger\|.
\]
Note that
\[
d_{\P}((A,\lambda),(A,\lambda'))=\arccos\frac{|1+\lambda\lambda'|}{\sqrt{1+|\lambda|^2}\sqrt{1+|\lambda'|^2}}\underset{\text{\eqref{eq:mmimi}}}{\leq} \frac{\epsilon}{5\mu(A,\lambda,v)},
\]
which implies
\[
|\lambda-\lambda'|\leq \frac{2\epsilon}{5\mu(A,\lambda,v)}.
\]
On the other hand,
\[
d_\P(v,v')\underset{\text{\eqref{eq:mmimi}}}{\leq} \frac{\epsilon}{5\mu(A,\lambda,v)}\rightarrow \alpha v'=v+r w,\text{ for some }w,\|w\|=1,r\leq \frac{\epsilon}{5\mu(A,\lambda,v)},|\alpha|=1,
\]
which implies
\[
\|vv^*-v'v'^*\|=\|r wv^*+r vw^*+r^2ww^*\|\leq 3r\leq \frac{3\,\epsilon}{5\mu(A,\lambda,v)}
\]
We thus have that
\[
\|(I- vv^*)(\lambda I-A)-(I- v'v'^*)(\lambda' I-A)\|\leq
\]
\[
\|(I- vv^*)(\lambda I-A)-(I- vv^*)(\lambda' I-A)\|+\|(I- vv^*)(\lambda' I-A)-(I- v'v'^*)(\lambda' I-A)\|\leq
\]
\[
|\lambda-\lambda'|+\|vv^*-v'v'^*\|\leq\frac{2\epsilon}{5\mu(A,\lambda,v)}+\frac{3\,\epsilon}{5\mu(A,\lambda,v)}\leq \
\frac{\epsilon}{\|((I- vv^*)(\lambda I-A))^\dagger\|},
\]
and from lemma \ref{lem:moore} we conclude that
\[
\frac{1}{1+\epsilon}\|((I- vv^*)(\lambda I-A))^\dagger\|\leq \|((I- v'v'^*)(\lambda' I-A))^\dagger\|\leq \frac{\|((I- vv^*)(\lambda I-A))^\dagger\|}{1-\epsilon},
\]
as claimed.
\end{proof}
\bibliographystyle{amsplain}

\bibliography{myreferences2}
\end{document}